\documentclass{amsart}

\usepackage{comment}
\usepackage{xcolor}
\usepackage{amssymb}
\usepackage{graphicx}
\usepackage{todonotes}

\newcommand{\priorrevisioncolor}{blue}
\newcommand{\amendmentcolor}{red}

\newcommand{\ri}{\mathrm{i}}

\newcommand{\diag}{\mathrm{diag}}

\theoremstyle{plain}
\newtheorem{theorem}{Theorem}[section]
\newtheorem{proposition}[theorem]{Proposition}
\newtheorem{lemma}[theorem]{Lemma}
\newtheorem{corollary}[theorem]{Corollary}

\theoremstyle{definition}
\newtheorem{definition}[theorem]{Definition} 
\newtheorem{example}[theorem]{Example}

\newtheorem{remark}[theorem]{Remark}

\begin{document}

\title[Resonant Spectral Submanifolds]{Existence and Regularity of Stable Resonant Spectral Submanifolds and Linearization Maps}
\author{Florian Kogelbauer}
\thanks{Corresponding Author: Florian Kogelbauer}
\address{Department of Mathematics, ETH Z\"{u}rich, R\"{a}mistrasse 101, 8092 Z\"{u}rich, Switzerland}
\email{floriank@ethz.ch}
\author{Rafael de la Llave}
\address{School of Mathematics, Georgia Institute of Technology, 30332-0160 Atlanta, Georgia}
\email{rafael.delallave@math.gatech.edu}

\keywords{Resonances, Spectral Submanifold, Linear Conjugacy, Logarithmic Polynomial, Polyhomogeneous Functions, Hartman Conjecture}

\subjclass[2020]{37D10, 37C05, 37C30}

\begin{abstract}
We prove the existence of $C^{r_{*},\alpha}$-regular sub-stable invariant manifolds
and conjugacies for analytic maps near a fixed point in the presence of resonances,
assuming an invertible semisimple linearization. The regularity order $r$ and the
admissible H\"older exponents $\alpha$ depend on the resonant indices and the relative
rates of contraction.\\
Our approach uses functional expansions based on logarithmic polynomials, which
allow the invariance equation to be solved explicitly and algorithmically to
arbitrary order. The existence results are obtained via a fixed-point argument
in a suitable Banach space. We further present several analytic examples that
illustrate the theorem and demonstrate the loss of regularity caused by resonances.
\end{abstract}

\maketitle

\section{Introduction}

Understanding the dynamics near an equilibrium is a classical problem in dynamical systems. Given a map or vector field with a fixed point, two closely related questions arise. The first is the regularity of a local change of variables that conjugates the nonlinear dynamics to its linear part. In the hyperbolic case, the Hartman--Grobman theorem gives a topological conjugacy \cite{grobman1959homeomorphism,hartman1960lemma,hartman1960local}, which prompts the more delicate question of how regular such a conjugacy can be. The second question is the existence and regularity of invariant manifolds tangent to invariant spectral subspaces of the linearization. In the parameterization method, these two questions can be formulated through closely related invariance equations.\\

More precisely, let $f:\mathbb{C}^s\to\mathbb{C}^s$ be a sufficiently regular map with $f(0)=0$, $A=Df(0)$ and let $X$ be an $A$-invariant subspace. In the parameterization method one seeks a parametrization $K$ and a reduced dynamics $R$ satisfying $f\circ K=K\circ R$. In this paper we study the more restrictive invariance-linearization equation obtained by prescribing the reduced dynamics to be the linear one, $R=A|_X$, so that
\begin{equation}\label{invariance_eq_intro}
        f\circ K=K\circ A|_X.
\end{equation}
We require $K$ to be tangent at the origin to the inclusion of $X$ into the ambient space. Thus, a solution simultaneously parametrizes an invariant manifold tangent to $X$ and linearizes the dynamics restricted to that manifold. Its regularity immediately implies the same regularity for the corresponding invariant manifold, but the converse need not hold: the image manifold may be smoother than a parametrization solving the invariance-linearization equation, because a loss of regularity may come from the additional requirement that the dynamics on the manifold be conjugated to its linearization rather than from the geometry of the manifold itself. This distinction is important below and is also the reason that it is useful to discuss invariant manifolds and the Hartman linearization problem together.\\

The purpose of the present paper is to understand the invariance--linearization problem \eqref{invariance_eq_intro} in the presence of resonances, in the sense of Definition \ref{def_resonances}. Rather than excluding the resonant orders at which the usual cohomological equations become singular, we enlarge the Taylor ansatz by logarithmic-polynomial terms and resolve the resulting equations recursively. The logarithmic approximation is corrected to an exact solution by a contraction in an analytic algebra of finitely many complete logarithmic monomials. Their regularity gives convergence in $C^{r_{*},\alpha}$ whenever $0<\alpha<\alpha_*$, with $\alpha_*$ computed directly from the selected spectrum. The weighted logarithmic-series spaces $\Gamma_N^d$, introduced in Section \ref{notation}, provide a complementary coefficient-space construction. The exponent $r$ is a spectrum-dependent lower guarantee for the invariance-linearization parametrization, not in general the optimal regularity of its image. If the first relevant resonant forcing coefficient vanishes, the associated logarithmic term disappears and higher regularity may occur. The convergent expansions we  present can be computed term by term and the corresponding calculations can be performed using  standard power-series manipulations. Keeping finitely many powers of $x$ gives any prescribed formal approximation and the coefficients remain polynomials in the logarithmic variables, while convergence holds in the H\"older spaces specified above. Locally, a tail beginning at degree $N$ satisfies $O(|x|^{N-\varepsilon})$ for each $\varepsilon>0$, but, in general the logarithmic contributions preclude the bound $O(|x|^N)$.\\
When all directions are stable, the resulting linearizing conjugacy is $C^{1,\alpha}$ for each $0<\alpha<\alpha_{*}$. Classical results of Hartman and subsequent refinements guarantee differentiable or $C^{1,\beta}$ linearization under contraction or spectral hypotheses, with $\beta$ tied to the regularity and spectral condition, see, in particular, \cite{hartman1960local,newhouse2017hartman,zhang2017differentiability}. With the analytic and semisimplicity assumptions imposed here, we obtain a positive H\"older range for the derivative in terms of the spectrum. \\

\subsection{Some Remarks on Previous Literature}\label{intro_literature}

The role of resonances is already visible in the classical linearization theorems. Sternberg and Chen showed that, under suitable non-resonance assumptions, a sufficiently smooth hyperbolic system is smoothly locally conjugate to its linear part \cite{chen1963equivalence,chen1965local,sternberg1957local}, see also the geometric formulation in \cite{banyaga1996cohomology}. In the analytic category, small divisors lead to arithmetic conditions of Siegel--Bruno type \cite{siegel1942iteration,bruno1971analytic}. With resonances, one generally expects a normal form rather than complete analytic linearization, since resonant terms cannot in general be removed by a near-identity change of variables \cite{arnold2012geometrical,carr2012applications}. The question of differentiability of a linearizing homeomorphism at the fixed point, one classical formulation of the \emph{Hartman conjecture}, was settled for $C^\infty$ hyperbolic diffeomorphisms by Guysinsky, Hasselblatt, and Rayskin \cite{revisionGuysinskyHasselblattRayskin2003}. Zhang, Lu, and Zhang \cite{zhang2017differentiability} subsequently proved differentiability at the fixed point for $C^1$ diffeomorphisms whose derivative is H\"older continuous there, in finite dimensions. Differentiability at that point is a weaker conclusion than $C^1$ regularity of the conjugacy throughout a neighborhood. Related developments include bunching and invariant-foliation regularity \cite{hasselblatt1994periodic,pugh1997holder} and differentiable linearization results of Belitskii, Samovol, Stowe, Van Strien, Newhouse, and others \cite{belitskii1978equivalence,revisionSamovol1979,stowe1988linearization,vanstrien1990smooth,newhouse2017hartman,zhang2017differentiability}. Resonant hyperbolic linearization has also been studied by Rayskin and by Bonckaert, Naudot, and Yang \cite{rayskin1998alpha,bonckaert2003linearization}. \\
A related question is the H\"older regularity of the conjugacy and its
inverse. Barreira and Valls \cite{revisionBarreiraValls2007} proved
H\"older linearization with a positive exponent controlled by the spectrum.
For exponents arbitrarily close to one, Rayskin
\cite{rayskin1998alpha} treated Hartman's three-dimensional resonant
example, and Tan \cite{revisionTan2000} obtained the same conclusion
for $C^{1,1}$ hyperbolic diffeomorphisms under additional spectral
inequalities that allow resonances. Belitskii's $C^1$ linearization
theorem \cite{belitskii1978equivalence} covers spectra without mixed
quadratic resonances in the eigenvalue moduli. Zhang and Zhang
\cite[Theorem~1]{revisionZhangZhang2016} treat the remaining resonant
cases: for each prescribed $0<\alpha<1$, a $C^\infty$ diffeomorphism
near a hyperbolic fixed point admits a local conjugacy to its linear
part that is $\alpha$-H\"older together with its inverse, including
mixed stable-unstable spectra, where the neighborhood may depend on
$\alpha$.\\
For invariant manifolds, the same cohomological mechanism appears, but the regularity of the invariant image must be distinguished from the regularity of a parametrization that also linearizes its internal dynamics. For non-resonant spectral subspaces, the parameterization method gives strong existence, regularity, and computational results \cite{Cab2003pam,Cab2003,CABRE2005444}, with related developments for difference equations and numerical implementations \cite{delaLlaveLomeli2012,gonzalez2022finite}. In the stable case, the spectral subspace associated with the least contracting eigenvalues gives the slow directions governing asymptotic approach to the fixed point \cite{Wayne97}. Several notions of slow invariant manifold coexist \cite{lorenz1986existence,lorenz1987nonexistence,lorenz1992slow}, and a precise theory under suitable non-resonance assumptions was developed in \cite{Llave97,Cab2003,CABRE2005444}. Internal resonances may be absorbed into a nonlinear reduced dynamics when only the invariant manifold is sought, whereas external resonances couple selected and transverse directions and can impose genuine regularity obstructions \cite{belitskii1978equivalence,homburg2006invariant}. Our results treat the corresponding invariance-linearization problem without excluding these external resonances. Here a \emph{spectral submanifold} passes through the fixed point,
is locally invariant, and has the dimension and tangent space of the
chosen eigenvector subspace. For the precise definition of resonances, we refer to 
Definition~\ref{def_resonances}. Our construction permits exact resonances for the chosen stable
subspace and provides explicit lower bounds on regularity, without
identifying the resulting manifold as a member of maximal smoothness.\\

Logarithmic expansions themselves, sometimes called $\Psi$-series, have a long history and appear in the work of Horn and in the classical theory of ordinary differential equations \cite{horn1896ueber,hille1997ordinary,gavrilov1992non}. Related concepts occur in invariant-manifold and normal-form problems, for instance in work of Chaperon \cite{chaperon1986c,CHAPERON_2004}, and products of polynomials and logarithms appear in \cite{Mourtada91,bonckaert2003linearization}. In geometric singular analysis, mixed power-logarithmic expansions are known as polyhomogeneous expansions \cite{grieser2001basics}. The present paper brings this viewpoint to resonant invariance equations along strictly stable spectral subspaces and combines it with a Banach-space convergence argument and a recursive computational scheme.\\
Mere differentiable existence in the selected stable directions already
follows by restricting a differentiable linearization of the contraction
on the local stable manifold, see \cite[Theorem~1.3 and Remark~1.4]{newhouse2017hartman}.
Bonckaert and Naudot \cite[Theorem~1.2]{revisionBonckaertNaudot2017}
prove Mourtada-type homeomorphic linearization for smooth hyperbolic
diffeomorphisms with semisimple derivative, expressing the conjugacy as
a smooth function of finitely many products of coordinates and
logarithmic powers. Their
Section~3.6.1 discusses the regularity gains possible by treating
resonant terms separately. Here analyticity and selected stability
give a convergent \emph{analytic} realization in complete logarithmic
monomials (Section~\ref{sec_main_proof}), together with
the explicit support-dependent $C^{k,\alpha}$ bound in
Corollary~\ref{cor_weighted_regularity}, allowing external resonances.
This bound is sufficient, not an optimality criterion or a claimed
improvement over all classical finite-smoothness criteria
\cite{belitskii1978equivalence,samovol2010new}.
The quantities appearing in the proof are similar to the stable inequalities in Samovol's $S(k)$
condition \cite{revisionSamovol1979}, as stated in
\cite[Proposition~2.2]{revisionRenYang2003}. Thus the integer
thresholds agree with that classical sufficient criterion and the present
construction also gives an analytic realization in complete logarithmic
monomials and explicit $C^{k,\alpha}$ bounds for selected stable subspaces.
There is also a direct connection with resonant forcing in mechanical systems \cite{nayfeh2024nonlinear}. In Lyapunov's expansions near an equilibrium \cite{liapounoff1907probleme}, substitution of lower-order terms produces forcing whose exponential time dependence is determined by products of eigenmodes. When such a product has the same exponential rate as another eigenmode, the inhomogeneous linear equation is resonantly forced and acquires a polynomial factor in time. In spatial coordinates this mechanism produces the logarithmic terms used in the subsequent analysis and can be regarded as the flow analogue of the singular cohomological equations at resonant multi-indices.\\

The contribution of the present paper can therefore be summarized as follows. We allow external resonances in the invariance-linearization equation and resolve each singular cohomological equation by explicit logarithmic monomials. We prove convergence to an exact solution through an analytic function of finitely many complete logarithmic monomials, and we derive the resonance- and spectrum-dependent regularity $C^{r_{*},\alpha}$ for all $0<\alpha<\alpha_{*}$, together with weighted derivative estimates for the finite logarithmic terms. 
Corollary~\ref{cor_weighted_regularity} strengthens this lower bound by using weighted degrees and choosing the logarithmic supports separately. The recursive construction is algorithmic to arbitrary finite order and is directly amenable to symbolic power-series manipulation. We have carried it out by hand in several examples, while a systematic software implementation is left for future work. The logarithmic class is substantially more structured than a generic $C^{r,\alpha}$ class and has the differentiability permitted by the chosen power exponents away from the logarithmic singular sets.\\
The restriction to linear reduced dynamics is a deliberate choice of our approach. Internal resonances can often be absorbed into a nonlinear reduced map if one seeks only an invariant manifold, but composition with a prescribed nonlinear reduced dynamics is not naturally compatible with the logarithmic Banach spaces used here. A finite-order logarithmic ansatz followed by a $C^r$ argument is possible in principle, but then one must balance smallness of the remainder against the regularity of the approximate solution. We leave this extension for future work. Likewise, the logarithmic class should not be identified with a generic $C^{r,\alpha}$ class: it carries additional algebraic information about the resonant variables and admits regularity properties not encoded by the single exponent $r$.\\

We use throughout the resonance terminology and the set $I_{\rm res}(\Lambda)$ from Definition \ref{def_resonances} in Section \ref{notation}. The main result is the following.

\begin{theorem}\label{mainthm}
Let $f:\mathbb{C}^s \to\mathbb{C}^s$ be an analytic map with $f(0)=0$, understood, when $f$ represents a real system, in complex eigencoordinates compatible with the underlying real structure. For a real map, $X_1$ is assumed invariant under the real-structure involution, and the parametrization is restricted to the corresponding real subspace. All differentiability assertions below are in the real sense. Let
\begin{equation}
A=Df(0),
\end{equation}
be semisimple and invertible.
Let $X_1$ be an $A$-invariant, $d$-dimensional subspace spanned by eigenvectors of $A$ associated to eigenvalues $\Lambda = (\lambda_1,\ldots,\lambda_d)$ such that
\begin{equation}\label{asslambda}
|\lambda_j|< 1,\quad 1\leq j\leq d.
\end{equation}
Assume first that $I_{\rm res}(\Lambda)\neq\varnothing$. For each resonant multi-index $n\in I_{\rm res}(\Lambda)$, define
\begin{equation}\label{def_r_star}
    r(n):=
\min\left\{
|n|-1,\,
\min_{j\in\operatorname{supp}(n)} n_j
\right\} \geq 1,\quad r_{*} :=
\min_{n\in I_{\rm res}(\Lambda)} r(n).
\end{equation}
 For each resonant
support, introduce the following quotient and its H\"older regularity index:
\begin{equation}\label{def_sigma_lambda}
  \Delta_\Lambda(n)
  =\frac{\sum_{j\in\operatorname{supp}(n)}n_j (-\log|\lambda_j|)}
          {\max_{j\in\operatorname{supp}(n)}(-\log|\lambda_j|)},
  \qquad n\in I_{\rm res}(\Lambda),
\end{equation}
\begin{equation}\label{def_alpha_star}
  \alpha_{*}\,:=\,\min\left\{1,\,
   \min_{n\in I_{\rm res}(\Lambda)}
   \bigl(\Delta_\Lambda(n)-r_{*}\bigr)\right\}>0.
\end{equation}
There is a neighborhood of zero in $X_1$ on which a parametrization $h$ belongs to every class $C^{r_{*},\alpha}$ with $0<\alpha<\alpha_{*}$ satisfying
\begin{equation}
        h(0)=0,\qquad Dh(0)=\iota_{X_1},\qquad f\circ h=h\circ A|_{X_1},
\end{equation}
where $\iota_{X_1}$ denotes the canonical inclusion embedding of $X_1$ into the ambient space. After restricting the domain, $h$ is an embedding and its image is tangent to $X_1$ and locally forward invariant.  The parametrization admits logarithmic-polynomial approximations to arbitrary finite order. The stated regularity is a uniform lower bound on the regularity of this invariance-linearization parametrization. The estimate gives $C^{r_{*},1-}$ regularity when every resonant index $n$ satisfies $\Delta_\Lambda(n)-r_{*}\geq1$.  If $I_{\mathrm{res}}(\Lambda)=\varnothing$, the parametrization can be chosen analytic.
\end{theorem}

\begin{corollary}[Weighted regularity and choice of logarithmic supports]
\label{cor_weighted_regularity}
Under the assumptions of Theorem~\ref{mainthm}, suppose
$I_{\rm res}(\Lambda)\ne\varnothing$. For
$\varnothing\ne S\subseteq\operatorname{supp}(n)$ put
\begin{equation}\label{def_support_degree}
 \Delta_\Lambda(n;S)
 =\frac{\sum_{j\in S}n_j(-\log|\lambda_j|)}{\max_{j\in S}(-\log|\lambda_j|)},
 \qquad D_\Lambda=\min_{n\in I_{\rm res}(\Lambda)}\Delta_\Lambda(n),
\end{equation}
and define
\begin{equation}\label{def_optimized_degree}
 D_\Lambda^{\rm opt}
 =\min_{n\in I_{\rm res}(\Lambda)}
   \max_{\varnothing\ne S\subseteq\operatorname{supp}(n)}
       \Delta_\Lambda(n;S)\geq D_\Lambda>1.
\end{equation}
The parametrization can be chosen on one neighborhood to belong to
every $C^{k,\alpha}$ with integer $k\geq1$, $0<\alpha<1$, and
$k+\alpha<D_\Lambda^{\rm opt}$. With the full logarithmic supports
used in Theorem~\ref{mainthm}, the same assertion holds with
$D_\Lambda$ in place of $D_\Lambda^{\rm opt}$.
In full stable dimension these are regularity bounds for a local
conjugacy. 
\end{corollary}

\begin{remark}
Theorem \ref{mainthm} is formulated for the strictly stable part of the spectrum as expressed in assumption \eqref{asslambda}. Of course, the same conclusions hold true if we consider the strictly unstable part of the spectrum and  $f^{-1}$ instead of $f$, i.e., assumption \eqref{asslambda} can be reversed.\\
\end{remark}

\begin{remark}
The semisimplicity assumption is used in this paper to keep the spectral decomposition and the logarithmic cohomology equations transparent. The same strategy is expected to extend to general invertible $A$ after replacing eigendirections by generalized eigenspaces and keeping track of the additional nilpotent terms. We do not use that extension in the present proof and therefore state the theorem only in the semisimple setting.
\end{remark}

As explained above, the full-dimensional choice $X_1=\mathbb{C}^s$ turns the invariance-linearization equation into the usual local linearization problem. Theorem \ref{mainthm} therefore yields the following stable Hartman statement under the hypotheses used here.

\begin{corollary}\label{cor_Hartman}
    Under the assumptions of Theorem \ref{mainthm}, if $X_1=\mathbb C^s$ and all eigenvalues are strictly stable, and $I_{\rm res}(\Lambda)\ne\varnothing$, then $f$ is locally $C^{1,\alpha}$-conjugate to its linearization for every $0<\alpha<\alpha_{*}$ from \eqref{def_alpha_star}. The conjugacy is in fact $C^{r_{*},\alpha}$. If the resonant index set is empty, the conjugacy can be chosen analytic.
\end{corollary}

Corollary \ref{cor_Hartman} is a stable, analytic, semisimple differentiable-linearization statement. It is useful to distinguish it from the various formulations historically grouped under the Hartman conjecture. Hartman's contraction theorem gives differentiable linearization and, in finite dimensions, $C^{1,\beta}$ regularity for some $\beta>0$, while later work treats weaker smoothness and more general hyperbolic spectral configurations \cite{hartman1960local,newhouse2017hartman,zhang2017differentiability}. For analytic maps with semisimple derivative, one conjugacy obtained by the logarithmic construction belongs to $C^{1,\alpha}$ for all $0<\alpha<\alpha_{*}$. Unequal contraction rates can make $\alpha_{*}$ strictly smaller than one. We do not claim here a corresponding result for mixed stable--unstable spectra.\\

Our proof has two steps. First, we solve the cohomological equations recursively to obtain a finite logarithmic-polynomial approximation beyond all resonant orders. Second, we regard its complete logarithmic monomials as independent auxiliary variables and correct the approximation by one analytic contraction. Substitution then gives both an exact solution and a convergent logarithmic expansion in every stated H\"older class.\\

The paper is organized as follows. Section \ref{notation} collects notation and introduces logarithmic polynomials and the Banach spaces used in the proof. Section \ref{prelim} formulates the invariance and linearization problem and discusses the role of outer resonances. Section \ref{approxsol} constructs approximate logarithmic-polynomial solutions recursively, while Section \ref{existence} proves the existence and regularity theorem by the fixed-point argument. Section \ref{examples} gives examples illustrating the hypotheses and sharpness, while Section \ref{conclusion} contains concluding remarks and extensions.

\section{Notation and Basic Definitions}\label{notation}

Throughout the paper, $s$ will be the dimension of the ambient space, while $d\leq s$ will be the dimension of an invariant subspace or manifold. We write $\mathbb N_0=\{0,1,2,\ldots\}$ for the nonnegative integers. For coordinate vectors and multi-indices, $1@j$ denotes the $j$-th unit vector in the relevant dimension. For a multi-index $n\in\mathbb{N}_0^d$, $n=(n_1,..,n_d)$, let
\begin{equation}
    |n| = \sum_{j=1}^d n_j,
\end{equation}
denote its order and let
\begin{equation}
    \operatorname{supp}(n) = \{j: n_j\neq 0\},
\end{equation}
denote its support.\\
For a vector $x=(x_1,...,x_d)$, we denote the $(d-1)$-dimensional vector whose $x_k$-component has been removed as $\underline{x}_k$. We denote the $d\times d$ diagonal matrix with elements $a_1,...,a_d$ as $\diag(a_1,...,a_d)$. The ball of radius $r$ in $\mathbb{R}^d$ is denoted as $B^r({\mathbb{R}^d})$, while the ball of radius $r$ in $\mathbb{C}^d$ is denoted as $B^r({\mathbb{C}^d})$. When the distinction between real and complex vector spaces are clear from the context, we simply write $B^r_d$.\\
We denote the $M$-dimensional complex strip of width $\sigma$ as 
\begin{equation}\label{def_strip}
    \mathfrak{T}_\sigma = \{z\in \mathbb{C}^M: |\Im(z_j)|\leq \sigma,\quad 1\leq j \leq M \}. 
 \end{equation}
Let $\alpha\in (0,1]$. A function $f:U\to\mathbb{R}^{d_2}$ defined on an open subset $U\subseteq\mathbb{R}^{d_1}$ is $\alpha$-H\"{o}lder continuous if
\begin{equation}
    |f(x)-f(y)|\leq C_U |x-y|^\alpha,\quad x,y\in U,
\end{equation}
for some constant $C_U$ depending on the domain. If $\alpha=1$, the function is Lipschitz continuous. We write $f\in C^{k,\alpha}(U,\mathbb{R}^{d_2})$ if $f$ is $k$-times continuously differentiable, all derivatives of orders zero through $k$ are bounded on $U$, and its $k$-th derivative is $\alpha$-H\"older continuous on $U$. We write $C^{r,1-}$ for the class of functions $C^{r,1-\varepsilon}$ for all $\varepsilon\in (0,1)$, i.e.,
\begin{equation}\label{def_C_1_minus}
    C^{r,1-} = \bigcap_{\varepsilon\in (0,1)} C^{r,1-\varepsilon}.
\end{equation}
While each $C^{r,\alpha}$ is a Banach space under the H\"{o}lder norm 
\begin{equation}\label{def_C_r_norm}
\|f\|_{C^{r,\alpha}(\Omega)}
=
\sum_{|\beta|\leq r}\|D^\beta f\|_{L^\infty(\Omega)}
+
\sum_{|\beta|=r}
\sup_{\substack{x,y\in\Omega\\ x\neq y}}
\frac{|D^\beta f(x)-D^\beta f(y)|}{|x-y|^\alpha},
\end{equation}
where $f:\Omega\to\mathbb{R}^{d_2}$ and $\Omega\subseteq\mathbb{R}^{d_1}$ is open, the space \eqref{def_C_1_minus} is a Fr\'echet space with the countable family of norms $\|\cdot\|_{C^{r,\alpha_m}(\Omega)}$, $\alpha_m=1-1/(m+1)$, $m\geq1$. All local regularity assertions below are understood after restriction to a sufficiently small ball.\\
Let $\mathcal{P}_{d,K}$ denote the space of polynomials in $d$ variables of degree at most $K$ for which  
\begin{equation}
\dim \mathcal{P}_{d,K} = \binom{K+d}{d}.
\end{equation}
We write
\begin{equation}
\mathcal{P}_d = \bigcup_{K = 0}^\infty  \mathcal{P}_{d,K}
\end{equation}
for the vector space of polynomials in $d$ variables.  \\
A \emph{power function} is a generalized polynomial expression $\mathfrak{q}:\mathbb{C}^d\to [0,\infty)$ of the form\footnote{We choose to denote power functions with Gothic fonts to cleanly distinguish them from polynomial expressions and avoid potential confusions.}
\begin{equation}\label{defpower}
    \mathfrak{q}(x) = \sum_{j=1}^d p_j  |x_j|^{\nu_j},
\end{equation}
where $\nu_j>0$ and $p_j\geq0$, with at least one $p_j>0$. If all the $\nu_j$ are positive even integers, power functions are polynomials in the real and imaginary parts of $x$, generally not holomorphic polynomials.\\
For a bounded linear operator $T\in\mathcal L(X,Y)$ between Banach spaces $(X,\|.\|_X)$ and $(Y,\|.\|_Y)$, we write
\begin{equation}
    \|T\|_{\rm op} = \sup_{\|x\|_{X}=1} \|Tx\|_{Y},
\end{equation}
for its operator norm. For $T\in\mathcal L(X):=\mathcal L(X,X)$ on a complex Banach space, we write
\begin{equation}
    \sigma(T) = \{\lambda \in \mathbb{C}: T-\lambda \operatorname{Id}_X \text{ is not bijective with a bounded inverse}\},
\end{equation}
for its spectrum. For a bounded linear operator on a real Banach space, its spectrum means the spectrum of its complexification.

\begin{definition}\label{def_resonances}
Given an $s\times s$ matrix $A$ over the complex numbers, let $\sigma(A)\subset\mathbb{C}$ denote its spectrum and $\lambda=(\lambda_1,\ldots,\lambda_s)$ its eigenvalues, counted with multiplicity. We call a multi-index $n\in\mathbb{N}_0^s$ with $|n|\geq 2$ \emph{resonant} if there exists an eigenvalue $\widetilde{\lambda}$ of $A$ such that
\begin{equation}\label{lambdares}
    \lambda^n=\widetilde{\lambda}.
\end{equation}
If a multi-index is not resonant, we call it \emph{non-resonant}.\\
Let $\Lambda=(\lambda_1,\ldots,\lambda_d)$ be the indexed list of eigenvalues selected by $X_1$. For a multi-index $n\in\mathbb{N}_0^d$ with $|n|\geq 2$, we write
\begin{equation}
    \lambda_\Lambda^n:=\lambda_1^{n_1}\cdots\lambda_d^{n_d}.
\end{equation}
We call $n$ \emph{resonant with respect to $\Lambda$} if there exists $1\leq k\leq s$ such that
\begin{equation}\label{lambdaresLambda}
    \lambda_\Lambda^n=\lambda_k.
\end{equation}
We regard $\Lambda=(\lambda_1,\ldots,\lambda_d)$ as an indexed list, with multiplicities, after choosing an eigenbasis adapted to $X_1$. The resonant pair $(n,k)$ is called an \emph{internal} or \emph{inner resonance} if $k\leq d$, and an \emph{external} or \emph{outer resonance} if $k>d$. Thus a repeated eigenvalue may give both kinds of resonance and the target eigendirection determines the distinction. For a given selected list $\Lambda$, we denote the set of all its resonant indices by
\begin{equation}\label{def_I_res}
    I_{\rm res}(\Lambda)
    =
    \left\{
        n\in\mathbb{N}_0^d :
        |n|\geq2,\quad
        \lambda_\Lambda^n=\lambda_k
        \text{ for some } 1\leq k\leq s
    \right\}.
\end{equation}
\end{definition}

\begin{lemma}\label{lemma_finite_resonances}
Under \eqref{asslambda} and the invertibility of $A$, the set $I_{\rm res}(\Lambda)$ is finite.
\end{lemma}
\begin{proof}
Set $\rho=\max_{1\leq j\leq d}|\lambda_j|<1$ and $m=\min_{\lambda\in\sigma(A)}|\lambda|>0$. If $n\in I_{\rm res}(\Lambda)$, then $m\leq|\lambda_\Lambda^n|\leq\rho^{|n|}$. Hence $|n|\leq (\log m)/(\log\rho)$, with the inequality interpreted in the usual way since $\log\rho<0$. Thus only finitely many multi-indices can be resonant.
\end{proof}

\begin{remark}
Complex eigencoordinates for a real analytic system are regarded coordinates
on its complexification. The real points are recovered as the fixed
locus of an antilinear involution: it conjugates coordinates belonging
to real eigenvectors and exchanges the coordinates of conjugate pairs.
The selected subspace must be preserved by this involution if it is to
represent a real manifold. Both the logarithmic variables and the
coefficient solutions are chosen accordingly, before restricting the
parametrization to the fixed locus. Note that its corresponding real dimension is $d$, not the
dimension $2d$ obtained by realifying the entire space $\mathbb C^d$.
Real two-by-two blocks for conjugate eigenvalues provide an equivalent
description. Differentiability of the logarithmic expressions always
refers to real derivatives.
\end{remark}

Coefficient vectors are measured in the $\ell^1$ norm of the fixed
ambient eigenbasis. For the spatial domain we also use an equivalent
maximum norm when convenient. Throughout, the assertions concern a
neighborhood of the origin.

\subsection{Logarithmic Series}

Before we state the problem of invariance, conjugacy and linearization, we introduce some classes of functions and appropriate Banach spaces associated with them. An expression of the form
\begin{equation}\label{deflogpoly}
\sum_{1\leq |n|\leq N} H_n(\log|\phi_1(x)|,...,\log|\phi_M(x)|)x^n,
\end{equation}
where $\phi_1,...,\phi_M$ are H\"{o}lder continuous functions in $x$ in a neighborhood of the origin and $H_n$ are polynomial functions is called \emph{logarithmic polynomial of degree $N$}. For $N=\infty$, an expression of the form \eqref{deflogpoly} is called \emph{logarithmic series}. The functions $\phi_j$ are called \emph{basis functions}, while $\eta_j:=\log|\phi_j|$ will be called the corresponding \emph{logarithmic variables}.

\begin{remark}
    Typically, it is assumed that the functions $\phi_j$ are at most of power-law growth at infinity,
\begin{equation}\label{polybound}
    \phi(x) = \mathcal{O}(|x|^\varkappa),\quad x\to\infty,
\end{equation}
for some $\varkappa>0$. In the following, however, we will be mostly interested in the behavior of functions of the form \eqref{deflogpoly} around the origin, so we do not include the asymptotic bound \eqref{polybound} in the definition of a logarithmic series. We stress, on the other hand, that the functions $\mathfrak{q}_j$ we are using in the explicit construction of approximate solutions to the invariance equation in the presence of resonances will be power functions of the form  \eqref{defpower}, which, of course, satisfy \eqref{polybound}. Let us also remark the assumption of $\phi_j$ being H\"{o}lder continuous is not substantial for the subsequent arguments. Again, using power functions for $\phi_j$ of course implies a certain regularity, which is at least H\"{o}lder continuous.  
 \end{remark}

For the fixed-point argument appearing later in the paper, we will need appropriate function spaces that reflect the regularity of solutions as well as proximity to approximate solutions to the invariance equation, which will be given in terms of logarithmic polynomials.\\
For $\beta,\sigma>0$, let $\mathcal A_{\beta,\sigma}$ consist of
functions holomorphic in the interior of $\mathfrak T_\sigma$,
continuous on its closure, with finite norm
\begin{equation}\label{defbetanorm}
 \|P\|_{\beta,\sigma}
 =\sup_{y\in\mathfrak T_\sigma}e^{-\beta|\Re y|}|P(y)|.
\end{equation}
Vector-valued functions use the standard $\ell^1$ norm in $\mathbb C^s$.
Weighted uniform convergence implies local uniform convergence, so
$\mathcal A_{\beta,\sigma}$ is complete. We also write
$C^\omega_{\beta,\sigma}(\mathbb C^M)=\mathcal A_{\beta,\sigma}$.
Consider the formal coefficient series
\begin{equation}\label{def_H}
 H=\sum_{|n|\geq1}H_n(\xi)x^n,
 \qquad H_n\in\mathcal A_{|n|\beta,\sigma}.
\end{equation}
At this stage $x$ and $\xi\in\mathbb C^M$ are independent
indeterminates. Substitution of $\xi(x)$, and convergence of the
resulting functions, will be justified later.
Define
\begin{equation}\label{defnorm}
 \|H\|_{\Gamma_N^d,\beta,\sigma}
 =\sum_{|n|\geq N}\|H_n\|_{|n|\beta,\sigma},
 \qquad H_n=0\quad (|n|<N),
\end{equation}
and
\begin{equation}\label{defGamma}
 \Gamma^d=\Gamma_1^d
 =\left\{H\text{ of the form \eqref{def_H}}:
       \sum_{|n|\geq1}\|H_n\|_{|n|\beta,\sigma}<\infty\right\}.
\end{equation}
Requiring all coefficients below order $N$ to vanish defines the
closed tail subspace $\Gamma_N^d$, with the norm inherited from
$\Gamma_1^d$. Thus the spaces are weighted $\ell^1$ Banach spaces.
They distinguish formal series even when substitution happens to give
the same function. The scalar space is a Banach algebra without a unit,
and the vector-valued space is its module. Since orders add under
multiplication,
$\Gamma_N^d\Gamma_L^d\subseteq\Gamma_{N+L}^d$ in the scalar case.
The same product estimate controls multilinear maps on vector series.\\
Indeed, we first note that for two scalar coefficient polynomials $p,q$ in $M$ variables and $\beta = \beta_1+\beta_2$, we have that
\begin{equation}
\begin{split}
       \|qp\|_{\beta,\sigma} & = \sup_{y\in\mathfrak{T}_\sigma} e^{-\beta |\Re y|} |p(y)q(y)|\\
       & = \sup_{y\in\mathfrak{T}_\sigma} e^{-\beta_1 |\Re y|} |p(y)| e^{-\beta_2 |\Re y|}|q(y)|\\
    & \leq \|p\|_{\beta_1,\sigma}\|q\|_{\beta_2,\sigma}.
\end{split}
\end{equation}
For scalar formal series $H$ and $\tilde H$, the product estimate is obtained by
\begin{equation}
    H(x) = \sum_{|n|\geq N} H_{n}(\xi) x^n,\quad \tilde{H}(x) = \sum_{|n|\geq N} \tilde{H}_{n}(\xi) x^n
\end{equation}
we have that 
\begin{equation}
\begin{split}
\|H\tilde{H}\|_{\Gamma_N^d,\beta,\sigma} & = \sum_{|k|\geq 2N}\left\|\sum_{n+m=k} H_{n}\tilde{H}_m\right\|_{\beta |k|,\sigma}\\
& \leq \sum_{|k|\geq 2N}\sum_{n+m=k} \|H_{n}\tilde{H}_m\|_{\beta |k|,\sigma} \\
& \leq  \sum_{|k|\geq 2N}\sum_{n+m=k} \|H_n\|_{\beta|n|,\sigma}  \|\tilde{H}_m\|_{\beta|m|,\sigma}  \\
 & \leq  \left(\sum_{|n|\geq N} \|H_n\|_{\beta|n|,\sigma}\right) \left(\sum_{|m|\geq N} \|\tilde{H}_m\|_{\beta|m|,\sigma}\right)\\
& \leq \|H\|_{\Gamma_{N}^d,\beta,\sigma}\|\tilde{H}\|_{\Gamma_{N}^d,\beta,\sigma},
\end{split}
\end{equation}
where we have used that $k=n+m$ allows us to split $\beta|k| = \beta|n|+\beta|m|$.

\begin{remark}
    The norm \eqref{defnorm} only measures the coefficient functions $H_n$, which are polynomial in the recursive construction but may be holomorphic in the completed space, and does not take into account the regularity of the composed map
\begin{equation}\label{Hncomplog}
        x\mapsto H_n(\log|\phi_1(x)|,...,\log|\phi_M(x)|).
    \end{equation}
Indeed, the regularity of the function \eqref{Hncomplog} and hence of sums and series expressions of the form \eqref{deflogpoly} depends on the specific choice of functions $\phi_j$, as detailed in the following section. Degree always counts the powers of $x$, independently of the logarithmic variables. To pass to differentiable functions we use the support restrictions on the constructed monomials and the finite-generator argument given below.  %We further note that a given function $f$ might admit several series expansions as written in full generality in \eqref{deflogpoly}, as shown by $x^2 = \log(e^x)x$. For the resolution of the cohomology equation in the presence of resonances, however, the basis functions will be of genuinely lower order compared to polynomials and thus the separation in a polynomial and a logarithmic component as measured by the norm \eqref{defnorm} is meaningful. 
\end{remark}

\section{Preliminaries and Formulation of the Problem of Linear Conjugacy}\label{prelim}

In this section, we specify the set-up and define the invariance-linearization equation. Appendix \ref{sec_appendix} establishes the corresponding result for vector fields, including partial parametrizations.

\subsection{Formulation of the Problem}

We will be concerned with analytic maps $f:\mathbb{C}^s\to\mathbb{C}^s$ such that $f(0)=0$, write
\begin{equation}\label{defA}
A=Df(0),
\end{equation}
for the linear part at zero, and set
\begin{equation}
f(x)=Ax+N_f(x),
\end{equation}
where $N_f(x)=\mathcal{O}(|x|^2)$. For a given map $f:\mathbb{C}^s\to\mathbb{C}^s$ and an $A$-invariant $d$-dimensional subspace $X_1\subseteq\mathbb{C}^s$, the general parameterization method seeks an embedding $K:B^1_d\to\mathbb{C}^s$ and a reduced map $R:B^1_d\to X_1$ satisfying $f\circ K=K\circ R$. In this paper we first impose the linear choice $R=A|_{X_1}$ and denote the corresponding parametrization by $h:B^1_d\to\mathbb{C}^s$, where $B^1_d\subset X_1$. Thus we consider the \emph{invariance-linearization equation}:
\begin{equation}\label{invariance}
f(h(x))=h(Ax),
\end{equation}
where $A=Df(0)$ and $Ax$ has to be read as $A|_{X_1}x$. Equation \eqref{invariance} states that, on the locally invariant image parametrized by $h$, regarded as a real differentiable manifold, the map can be conjugated to its linearization at zero. Equation \eqref{invariance} simultaneously imposes invariance of the image of $h$ and linearization of the reduced dynamics on that image and is hence called the \emph{invariance-linearization equation}.\\
Inner and outer resonances are understood in the sense of Definition \ref{def_resonances}. Inner resonances do not substantially affect the regularity of the invariant manifold and can be resolved with nonlinear terms in the conjugated dynamics \cite{CABRE2005444}. Indeed, allowing for more general reduced dynamics 
\begin{equation}
    f(h(x)) = h(g(x)),
\end{equation}
where 
\begin{equation}\label{nonlinearreduced}
    g(x) = A|_{X_1}x + \mathcal{O}(|x|^2),
\end{equation}
allows for nonlinear contributions and thus resolves inner resonances. Outer resonances can limit the regularity of invariant manifolds when the associated forcing does not vanish. This obstruction is crucial in certain normal form calculations \cite{Neild20140404}.\\
In the following, the distinction between inner and outer resonances will be softened. The regularity guarantee for the invariance-linearization parametrization combines the relative contraction rates with every selected resonant index. Internal and external target directions enter this estimate in the same way, while the image itself can be smoother. If the minimal resonant index is, indeed, an inner resonance, the regularity of the invariant manifold could be improved by considering a non-linear reduced dynamics \eqref{nonlinearreduced}.\\

We work in complex eigenvector coordinates. The semisimple hypothesis
therefore gives diagonal matrices and removes the need for generalized
eigenvectors. Analyticity of the map does not in general make the
logarithmic parametrization analytic: its real differentiability can
be finite. Note that an analytic solution may still occur when resonant forcing
vanishes. As discussed before, we formulate our results for maps - the relation to differential equations is well-known and treated in Appendix \ref{sec_appendix}.

\section{Approximate Solution to the Invariance Equation by Logarithmic Series}\label{approxsol}

In this section, we construct approximate solutions to the invariance equations in terms of logarithmic polynomials. These will serve as a basis for the fixed-point argument in a later section. First, we prove some results on the invertibility of certain shift operators on the space of polynomials. Then, we introduce a special choice of basis function in the logarithmic series \eqref{deflogpoly}, namely the class of power functions \eqref{defpower}, and discuss their regularity properties. Finally, we show that the invariance equation \eqref{invariance} can be solved to any order for this choice of logarithmic series.\\

\subsection{Shifted Polynomials}

Before we can construct approximate solutions to the invariance equation in terms of logarithmic series, we need some preparatory results on shifted polynomials. We emphasize already at this point that Lemma \ref{shifted} gives a \emph{unique} solution in the absence of resonances, while it only guarantees the existence of \emph{a} solution in the presence of resonances. Indeed, the lack of uniqueness in the presence of outer resonances is due to this existence of a kernel of a certain operator acting on polynomials, as elaborated in the following. Several logarithmic variables can occur at higher orders, so we formulate the following polynomial shift lemma in arbitrary finite dimension. \\
For $a\in\mathbb{C}$ and $v\in\mathbb{R}^d$, let $T_v:\mathcal{P}_d\to\mathcal{P}_d$ denote the shift operator
\begin{equation}\label{def_shift_v}
    T_v[p](x)=p(x+v),
\end{equation}
and define the operator $S_{a,v}:\mathcal{P}_d\to\mathcal{P}_d$, $S_{a,v}=I-aT_v$, i.e.,
\begin{equation}\label{defSav}
    S_{a,v}[p](x) = p(x)- a p(x+v). 
\end{equation}
We observe that the operator norm of $T_v:\mathcal{P}_{d,n}\to\mathcal{P}_{d,n}$, where $\mathcal{P}_{d,n}$ is endowed with the norm $\|.\|_{\beta,\sigma}$, is bounded as 
\begin{equation}\label{boundnormTv}
    \|T_v\|_{\rm op} \leq e^{\beta |v|},
\end{equation}
independently of the degree of $p$. Indeed, we calculate
\begin{equation}
    |p(y+v)|e^{-\beta |\Re y|} = \Big( |p(y+v)|e^{-\beta |\Re y+v|} \Big)e^{\beta (|\Re y+v|-|\Re y|)},
\end{equation}
and using that $e^{\beta (|\Re y+v|-|\Re y|)}\leq e^{\beta |v|}$ by the triangle inequality, we can estimate
\begin{equation}
    \begin{split}
        \|T_v p\|_{\beta,\sigma} & = \sup_{y\in\mathfrak{T}_\sigma}  |p(y+v)|e^{-\beta |\Re y|}\\
        & \leq \Big(\sup_{y\in\mathfrak{T}_\sigma}  |p(y+v)|e^{-\beta |\Re y+v|}\Big) e^{\beta |v|}\\
        & \leq \|p\|_{\beta,\sigma} e^{\beta|v|}
    \end{split}
\end{equation}
which immediately implies the bound \eqref{boundnormTv}.
\begin{lemma}\label{shifted}
 Consider the operator $S_{a,v}:\mathcal{P}_{d,K}\to\mathcal{P}_{d,K}$ equipped with the norm $\|.\|_{\beta,\sigma}$ from \eqref{defbetanorm} and assume that $v\neq 0$. For $a\neq 1$, the spectrum of $S_{a,v}$ and its inverse are given by 
 \begin{equation}\label{sigmaSav}
 \sigma(S_{a,v}) = \{1-a\},\quad  \sigma(S_{a,v}^{-1}) = \left\{\frac{1}{1-a}\right\},
 \end{equation}
 independently of the order of the polynomials $K$ and the shift vector $v$.\\
 In particular, $S_{a,v}:\mathcal{P}_{d,K}\to\mathcal{P}_{d,K}$ is invertible for $a\neq 1$ and we can bound its inverse as 
     \begin{equation}\label{boundSavinvfull}
    \|S_{a,v}^{-1}\|_{\rm op} \leq \frac{1}{|1-a|}\sum_{j=0}^{K}(q\Delta)^j,
\end{equation}
where we have set
\begin{equation}
    q = \frac{|a|}{|1-a|},\quad \Delta =\|T_v-I\|_{\rm op}.
\end{equation}
If in addition
\begin{equation}\label{boundaevb}
   |a| e^{\beta|v|}<1, 
\end{equation}
we obtain the bound
\begin{equation}\label{boundSavinv}
    \|S_{a,v}^{-1}\|_{\rm op} \leq \frac{1}{1- |a| e^{\beta|v|}},
\end{equation}
while for
\begin{equation}
|a|^{-1} e^{\beta |v|}<1,
\end{equation}
we obtain the bound
\begin{equation}\label{boundSavbinv2}
\|S_{a,v}^{-1}\|_{\rm op}\leq \frac{1}{|a|e^{-\beta|v|}-1}.
\end{equation}
For $a=1$ and $v\neq 0,$ the operator $S_{a,v}:\mathcal{P}_{d,K+1}\to\mathcal{P}_{d,K}$ is surjective. 
\end{lemma}
\begin{proof}
Let $D_v=v\cdot\nabla$ denote the streaming operator with velocity $v$, i.e., the generator of $T_v$. The differential operator $D_v:\mathcal{P}_{d,K}\to\mathcal{P}_{d,K}$ is well-defined and satisfies $D_v^{K+1}=0$ and hence
\begin{equation}
    T_v = \exp(D_v) = \sum_{n=0}^K \frac{D^n_v}{n!}.
\end{equation}
The operator $T_v:\mathcal{P}_{d,K}\to\mathcal{P}_{d,K}$ is thus the sum of the identity plus a nilpotent operator, and hence an unipotent operator for every $K$. This implies that $\sigma(T_v)=\{1\}$ on each finite-dimensional space $\mathcal P_{d,K}$ and \eqref{sigmaSav} follows.\\
Writing
\begin{equation}
    S_{a,v} = (1-a)\left(I-\frac{a}{1-a}(T_v-I)\right),
\end{equation}
and noting that the operator $T_v-I$ is nilpotent of order $K+1$, i.e., $(T_v-I)^{K+1}=0$, implies that we can write a closed-form expression for the inverse of $S_{a,v}$:
\begin{equation}\label{inversexplicit}
    S_{a,v}^{-1} = \frac{1}{1-a} \sum_{k=0}^K \left(\frac{a}{1-a}\right)^k(T_v-I)^k.
\end{equation}
This immediately implies the bound 
\begin{equation}
    \|S_{a,v}^{-1}\|_{\rm op} \leq \frac{1}{|1-a|}\sum_{k=0}^K \left(\frac{|a|}{|1-a|}\right)^{k} \|T_v-I\|_{\rm op}^k,
\end{equation}
which is \eqref{boundSavinvfull}. For $q\Delta\ne1$ the sum equals $((q\Delta)^{K+1}-1)/(q\Delta-1)$, and for $q\Delta=1$ it equals $K+1$.\\ 
Assuming now \eqref{boundaevb}, we can invert $S_{a,v}$ by a Neumann series and the estimate \eqref{boundSavinv} follows right away. The other estimate \eqref{boundSavbinv2} follows from rewriting $S_{a,v}p=q$ as 
\begin{equation}
\frac{1}{a}p(x-v) -p(x) = \frac{1}{a}q(x-v). 
\end{equation}
To prove the surjectivity of $S_{1,v}:\mathcal{P}_{d,K+1}\to\mathcal{P}_{d,K}$, let us first note that we can factor
\begin{equation}
\begin{split}
       S_{1,v} &  = 1 - \exp(D_v)\\
       & =  1 - \sum_{n=0}^{K+1}\frac{D^n_v}{n!} \\
       & = -D_v\left(\sum_{n=0}^{K}\frac{D_v^n}{(n+1)!}\right)\\
      & =: D_v \Sigma_v,\qquad \Sigma_v:=-\sum_{n=0}^{K}\frac{D_v^n}{(n+1)!},
       \end{split}
\end{equation} 
since $D_v^{K+2}=0$ on $\mathcal{P}_{d,K+1}$. The operator $-\Sigma_v$ is the identity plus a nilpotent operator and is therefore invertible - hence $\Sigma_v$ is invertible as well. It follows that
\begin{equation}
    \text{range}(S_{1,v}) = \text{range}(D_v\Sigma_v) = \text{range}(D_v). 
\end{equation}
Surjectivity of $S_{1,v}:\mathcal{P}_{d,K+1}\to\mathcal{P}_{d,K}$ is thus equivalent to surjectivity of the streaming operator $D_v:\mathcal{P}_{d,K+1}\to\mathcal{P}_{d,K}$. The kernel of $D_v$ consists exactly of those polynomials which are constant along the $v$-direction, i.e., polynomials in the $d-1$ independent coordinates transverse to $v$. Hence
\begin{equation}
    \dim \ker D_v = \binom{K+d}{d-1}. 
\end{equation}
Therefore, the rank of $D_v$ is given by
\begin{equation}
\begin{split}
    \text{rank } D_v & = \dim \mathcal{P}_{d,K+1} - \dim \ker D_v\\
     & = \binom{K+1+d}{d} - \binom{ K+d}{ d-1}\\
     & = \binom{K+d}{d},
    \end{split}
\end{equation}
where in the last step, we have used Pascal's identity for binomial coefficients and for the dimension of the space of polynomials we refer to \cite{cox2015ideals}. Consequently,  $\text{rank } D_v = \dim \mathcal{P}_{d,K} $ and $D_v$ is surjective. 
\end{proof}

\subsection{Power Functions as Basis Functions}

To resolve the invariance equation in the presence of resonances, we introduce series expansions of the form
\begin{equation}\label{expandheta}
    h(x) = \sum_{|n|=1}^\infty h_n(\xi) x^n,
\end{equation}
where each coefficient $h_n$ is a polynomial in the logarithmic variables $\xi(x)=(\eta_1(x),\ldots,\eta_M(x))$. Here $M$ is the number of logarithmic variables introduced by the resonant construction below, and
\begin{equation}\label{def_eta}
    \eta_l(x) = \log |\mathfrak{q}_l(x)|,\quad \mathfrak{q}_l(x) = \sum_{j=1}^d p_{l,j}|x_j|^{\nu_j},
\end{equation}
for exponents $\nu_j>0$ and weights $p_{l,j}\geq0$ chosen below. For each $l$, at least one $p_{l,j}$ is positive; later the support of the vector $(p_{l,1},\ldots,p_{l,d})$ is chosen to match the support required by the corresponding resonant multi-index.
To realize \eqref{expandheta} within the class \eqref{deflogpoly},
fix a positive scale $C$, set $\gamma=e^{-C}$, and choose
\begin{equation}\label{nu_form}
 \nu_j=\frac{\log\gamma}{\log|\lambda_j|}=\frac{C}{-\log|\lambda_j|},
 \qquad -\log|\lambda_j|>0.
\end{equation}
For any choice of the nonnegative weights in \eqref{def_eta},
$\mathfrak{q}_l(Ax)=\gamma\mathfrak{q}_l(x)$, and consequently
\begin{equation}\label{func_eq_eta}
 \xi(Ax)=\xi(x)+v,\qquad
 v=(-C,\ldots,-C)\in\mathbb R^M.
\end{equation}
Thus each component has the same scalar shift
\begin{equation}\label{defvlambda}
 v_l=\log\gamma=-C.
\end{equation}
A single shift suffices here, although independently chosen bases
need not share one. In the resonant case we increase $C$ until
$\nu_j>r_{*}+2$ for every $j$. This adjustment leaves all exponent
ratios, and hence the threshold in \eqref{def_alpha_star}, unchanged.

\begin{remark}
To solve the cohomology equation with resonances, it is the logarithmic variables $\eta_j=\log|\phi_j|$, rather than the basis functions $\phi_j$ themselves, that must satisfy an additive shift relation
\begin{equation}\label{fcteqeta}
    \eta_j(Ax)=v_j+\eta_j(x),\quad 1\leq j\leq M,
\end{equation}
for some real numbers $v_j\neq0$. Such an $\eta_j$ cannot extend continuously with a finite value to the origin: formally evaluating the relation at $x=0$ would give $\eta_j(0)=v_j+\eta_j(0)$. This is precisely why we write $\eta_j=\log|\phi_j|$ with simpler basis functions $\phi_j$ that vanish at the relevant singular set. In the construction below we take $\phi_j=\mathfrak{q}_j$, where $\mathfrak{q}_j$ is a power function of the form \eqref{defpower}. Then $\eta_j=\log\mathfrak{q}_j$ satisfies \eqref{fcteqeta}. This convention keeps the continuous basis functions $\phi_j$ distinct from their singular logarithmic variables $\eta_j$.

\end{remark}

\begin{lemma}\label{lemma_regularity}
Let $\mathfrak q$ be a power function as in \eqref{defpower},
with support $S=\{j:p_j>0\}$. Let $n\in\mathbb N_0^d$ satisfy
$|n|\geq2$ and
\begin{equation}\label{asslemma}
  S\subseteq\operatorname{supp}(n).
\end{equation}
Define
\begin{equation}\label{regularity_index}
 r=\min\left\{\sum_{j\in S}n_j-1,\ \min_{j\in S}n_j\right\},
\end{equation}
and assume
\begin{equation}\label{ass_nu_large}
 \nu_j>r+1\qquad(j\in S).
\end{equation}
Write $\nu=(\nu_1,\ldots,\nu_d)$ and set
$$
 \nu_{\min}=\min_{j\in S}\nu_j,\qquad
 \Delta(n,\nu)=\nu_{\min}\sum_{j\in S}\frac{n_j}{\nu_j},\qquad
 \alpha_{*}=\min\{1,\Delta(n,\nu)-r\}.
$$
Then $\Delta(n,\nu)>r$, and for every $Q\in\mathcal P_1$ the function
$$
 u(x)=x^n Q\bigl(\log\mathfrak{q}(x)\bigr),\qquad \mathfrak{q}(x)>0,
$$
extended by zero to
$\Sigma=\{x:x_j=0\text{ for all }j\in S\}$, belongs locally to
$C^{r,\alpha}$ for every $0<\alpha<\alpha_{*}$. If
$\Delta(n,\nu)>r+1$, the extension is in fact $C^{r+1}$. The same assertions hold with $r$ replaced by any integer
$0\leq k<\Delta(n,\nu)$, provided $\nu_j>k+1$ for $j\in S$,
and with $\alpha_{*}=\min\{1,\Delta(n,\nu)-k\}$.\\

For $M\geq1$, take power functions $\mathfrak q_l$ and logarithmic
variables $\eta_l$ of the form \eqref{def_eta}, with
$\varnothing\ne S_l=\{j:p_{l,j}>0\}\subseteq\operatorname{supp}(n)$
for $1\leq l\leq M$, and set
\begin{align*}
 \nu_{\min,l}=\min_{j\in S_l}\nu_j,\quad 
 \Delta_l(n,\nu)=\nu_{\min,l}\sum_{j\in S_l}\frac{n_j}{\nu_j},\quad 
 \Delta_{\min}(n,\nu)&=\min_{1\leq l\leq M}\Delta_l(n,\nu).
\end{align*}
If $k\geq0$ is an integer, $\nu_j>k+1$ on the union of the supports,
and $k<\Delta_{\min}(n,\nu)$, then, for every $Q\in\mathcal P_M$,
$$
 u(x)=x^nQ\bigl(\xi(x)\bigr)
$$
extends by zero to $Z=\bigcup_l\{x:\mathfrak{q}_l(x)=0\}$ and belongs
locally to $C^{k,\alpha}$ for every
$0<\alpha<\min\{1,\Delta_{\min}(n,\nu)-k\}$.
Its derivatives through order $k$ vanish on $Z$; if
$\Delta_{\min}(n,\nu)>k+1$, the extension is $C^{k+1}$.
In both cases the neighborhood can be fixed independently of the
admissible $\alpha$.
\end{lemma}

\begin{proof}
First, assume $\Delta(n,\nu)>r$: for $S=\{j\}$ one has
$\Delta(n,\nu)=n_j=r+1$, while otherwise, choosing $j_0\in S$ with
$\nu_{j_0}=\nu_{\min}$ gives
$$
 \Delta(n,\nu)=n_{j_0}+\sum_{j\in S\setminus\{j_0\}}
 \frac{\nu_{\min}}{\nu_j}n_j>n_{j_0}\geq r.
$$
It therefore suffices to prove the several-support assertion and
$M=1$ gives the single-support conclusions, including $k=r$.
All derivatives are real, with $x_j=a_j+\ri b_j$ and for a real
multi-index $m$, let $m_j$ count derivatives in $a_j$ and $b_j$.
Assume $Q\not\equiv0$ and write
$\Delta_{\min}=\Delta_{\min}(n,\nu)$.
On a fixed ball where $\mathfrak q_l\leq1$ for all $l$, set
$$
 d_l=\mathfrak q_l^{1/\nu_{\min,l}},\qquad
 d_0= \min_{1\leq l\leq M}d_l.
$$
For $j\in S_l$ and $1\leq q\leq k+1$, the assumption
$\nu_j>k+1$ gives
$$
 |D_{x_j}^{q}(p_{l,j}|x_j|^{\nu_j})|
 \leq C_q|x_j|^{\nu_j-q}
 \leq C_q\mathfrak q_l^{1-q/\nu_j}.
$$
This estimate holds also at $x_j=0$ by continuity. The chain rule yields
$$
|D^m\eta_l|
 \leq C_m\mathfrak q_l^{-\sum_{j\in S_l}m_j/\nu_j}
 \leq C_m d_l^{-|m|},\qquad
 1\leq|m|\leq k+1,\quad \mathfrak q_l>0.
$$
Moreover, for a first real derivative in coordinate $j\in S_l$,
$$
 |D_{x_j}d_l|\leq
 C\mathfrak q_l^{1/\nu_{\min,l}-1/\nu_j}\leq C,
$$
and derivatives in other coordinates vanish. Each $d_l$ is continuous
and zero on the linear subspace $\{\mathfrak q_l=0\}$. Integrating
along line segments, split at that subspace when necessary, shows
that $d_l$, and hence $d_0$, is Lipschitz. Let $L\geq1$ be a
Lipschitz constant for $d_0$.
At a point off $Z$, choose $l$ with $d_l=d_0$. A nonzero
$D^\beta(x^n)$ has coordinate counts $\beta_j\leq n_j$, so
$$
 |D^\beta(x^n)|
 \leq C_\beta d_0^{\nu_{\min,l}
                   \sum_{j\in S_l}(n_j-\beta_j)/\nu_j}
 \leq C_\beta d_0^{\Delta_{\min}-|\beta|}.
$$
Coordinates outside $S_l$ contribute bounded factors. Since
$d_i\geq d_0$ and $|\eta_i|\leq C|\log d_0|$,
Leibniz' rule and the logarithmic derivative estimate give
\begin{equation}\label{multilog_derivative}
 |D^m u|\leq C_m d_0^{\Delta_{\min}-|m|}(1+|\log d_0|)^{\deg Q},
 \qquad |m|\leq k+1,
\end{equation}
off $Z$. For $M=1$, retaining the coordinate weights in the same
calculation also gives the sharper estimate
\begin{equation}\label{corrected_weighted_derivative}
 |D^m u(x)|\leq C_m(1+|\log\mathfrak q(x)|)^{\deg Q}
 \mathfrak q(x)^{\sum_{j\in S}(n_j-m_j)/\nu_j},
 \qquad |m|\leq k+1,
\end{equation}
and, since $\nu_j\geq\nu_{\min}$,
\begin{equation}\label{corrected_sigma_bound}
 |D^m u(x)|\leq C_m(1+|\log\mathfrak q(x)|)^{\deg Q}
 \mathfrak q(x)^{(\Delta(n,\nu)-|m|)/\nu_{\min}}.
\end{equation}
Choose $0<\epsilon<\Delta_{\min}-k$ and absorb the logarithm in
\eqref{multilog_derivative}. For every $z\in Z$,
$d_0(x)\leq L|x-z|$, so derivatives through order $k$ tend to zero
at $Z$, and those through order $k-1$ are $o(|x-z|)$.
Extending them by zero and applying the definition of the derivative
inductively gives $u\in C^k$, including at intersections of the zero sets.\\
For $g=D^m u$, $|m|=k$, and
$0<\alpha<\min\{1,\Delta_{\min}-k\}$,
\eqref{multilog_derivative} similarly gives
\begin{equation}\label{corrected_holder_pair}
 |g|\leq C_\alpha d_0^\alpha,\qquad
 |\nabla g|\leq C_\alpha d_0^{\alpha-1}\quad\text{off }Z.
\end{equation}
The first bound holds also on $Z$. For distinct $x,y$ in a fixed
smaller ball, put $h=|x-y|$. If
$\min\{d_0(x),d_0(y)\}\leq2Lh$, both scales are at most $3Lh$,
and the first bound gives $|g(x)-g(y)|\leq C_\alpha h^\alpha$.
Otherwise $d_0\geq Lh$ along the segment from $x$ to $y$, so
the gradient bound and the mean-value inequality give the same estimate.
This proves the stated H\"older regularity on a ball independent of
$\alpha$, without differentiating $d_0$.\\
Finally, if $\Delta_{\min}>k+1$, choose
$0<\epsilon<\Delta_{\min}-(k+1)$ in
\eqref{multilog_derivative}. Derivatives through order $k$ are then
$o(|x-z|)$ at $z\in Z$, and those of order $k+1$ extend continuously
by zero. The same induction gives $u\in C^{k+1}$, using only
$\nu_j>k+1$.
\end{proof}

\begin{remark}[A restriction imposed by unequal coordinate scales]
The hypotheses \eqref{asslemma} and \eqref{ass_nu_large} do not
by themselves give $C^{r,1-}$ regularity. This can be seen from
the following function on $\mathbb R^2$:
$$
 \mathfrak{q}({u},{v})={u}^4+{v}^8,\qquad
 \mathcal U({u},{v})={u}{v}\log({u}^4+{v}^8).
$$
The monomial has $n=(1,1)$ and $r=1$, while the powers satisfy
$\nu_1=4,\nu_2=8>2$. The quantities in
Lemma~\ref{lemma_regularity} are $\nu_{\min}=4$ and
$\Delta(n,\nu)=4(1/4+1/8)=3/2$.
For $t>0$ put $P_t=(t,t^{1/2})$ and $Q_t=(0,t^{1/2})$.
A direct calculation yields
$$
 |P_t-Q_t|=t,\qquad
 \bigl|\partial_{u}\mathcal U(P_t)-\partial_{u}\mathcal U(Q_t)\bigr|
 =(\log 2+2)t^{1/2}.
$$
The displayed increment excludes $\mathcal U\in C^{1,\alpha}$ when
$\alpha>1/2$, whereas Lemma~\ref{lemma_regularity} gives membership
for every $\alpha<1/2$. The threshold is therefore sharp for this
example. A common positive multiple of the powers $\nu_j$ leaves
$\Delta(n,\nu)$ unchanged, so it cannot remove the loss caused by
the unequal scales.
\end{remark}

In the following, we collect some counterexamples as remarks that show that the assumptions of Lemma \ref{lemma_regularity} cannot be relaxed in general. 

\begin{remark}
Functions of the form $x\mapsto x^n Q(\log\mathfrak{q}(x))$ are generally not $C^{N,\alpha}$ for $N\geq |n|$. Indeed, along the real line $x(t)=(t,\dots,t)$, $t>0$, we have that 
\begin{equation}
    \mathfrak{q}(x(t))=t^{\nu_{\min}}(c+o(1)),\qquad
x(t)^n=t^{|n|}, 
\end{equation}
where $\nu_{\min}=\min_{j:p_j>0}\nu_j$ and $c=\sum_{j:\nu_j=\nu_{\min}}p_j>0$, and thus
\begin{equation}
x^n\log\mathfrak{q}(x(t))=\nu_{\min} t^{|n|}\log t+O(t^{|n|}).
\end{equation}
The function $t^{|n|}\log t$ is $C^{|n|-1}$ but not $C^{|n|}$ at $0$.  This shows that one cannot in general claim $C^{|n|}$ regularity for nonconstant logarithmic factors. The H\"older exponent below the integer differentiability threshold is instead restricted by the weighted index $\Delta(n,\nu)$, as the preceding example demonstrates.
\end{remark}

\begin{remark}
The lower bound on the exponents $\nu_j$ in \eqref{ass_nu_large} is used to control the indicated number of real derivatives of the basis functions. If it is dropped, further H\"older losses can occur, independently of the anisotropy described above. Consider the real two-dimensional example
\begin{equation}
        n=(1,1),
        \qquad
        \mathfrak{q}(x,y)=|x|^\varkappa+|y|,
        \qquad
        u(x,y)=xy\log\mathfrak{q}(x,y),
\end{equation}
for $0<\varkappa<1$. The support condition \eqref{asslemma} is satisfied and the
regularity index predicted by the preceding lemma is $r=1$. For $x>0$, we have
\begin{equation}
        \partial_x u(x,y)
        =
        y\log(x^\varkappa+|y|)
        +
        \varkappa y\frac{x^\varkappa}{x^\varkappa+|y|}.
\end{equation}
At $x=0$, the continuous extension of this derivative, if it exists, must be
\begin{equation}
        \partial_x u(0,y)
        =
        y\log |y|.
\end{equation}
Fixing $y=y_0>0$, we obtain as $x\to0^+$
\begin{equation}
\begin{split}
        \partial_x u(x,y_0)-\partial_x u(0,y_0)
        &=
        y_0\log\left(1+\frac{x^\varkappa}{y_0}\right)
        +
        \varkappa y_0\frac{x^\varkappa}{x^\varkappa+y_0}  \\
        &=
        (1+\varkappa)x^\varkappa+o(x^\varkappa).
\end{split}
\end{equation}
Consequently, $\partial_xu$ is not $\alpha$-H\"older continuous in the $x$-direction
for any $\alpha>\varkappa$. In particular, choosing $\varepsilon>0$ so small that $1-\varepsilon>\varkappa$, we see that $u\notin C^{1,1-\varepsilon}$.  This example does not satisfy \eqref{ass_nu_large} and demonstrates why the smoothness of the power-function terms must be controlled separately, see also Figure \ref{small_nu_plot}. \\
\begin{figure}
    \centering
\includegraphics[width=0.45\linewidth]{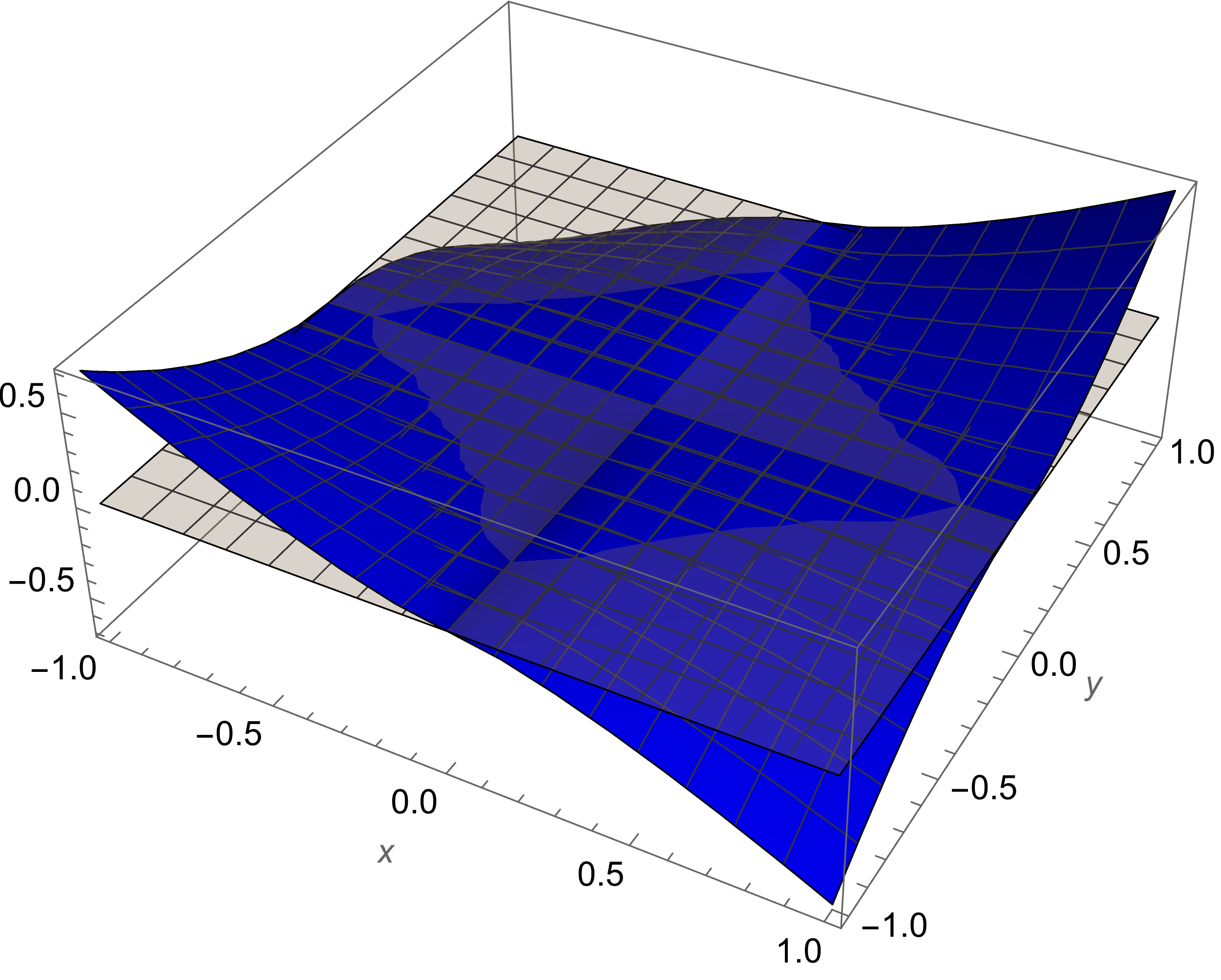}
\includegraphics[width=0.45\linewidth]{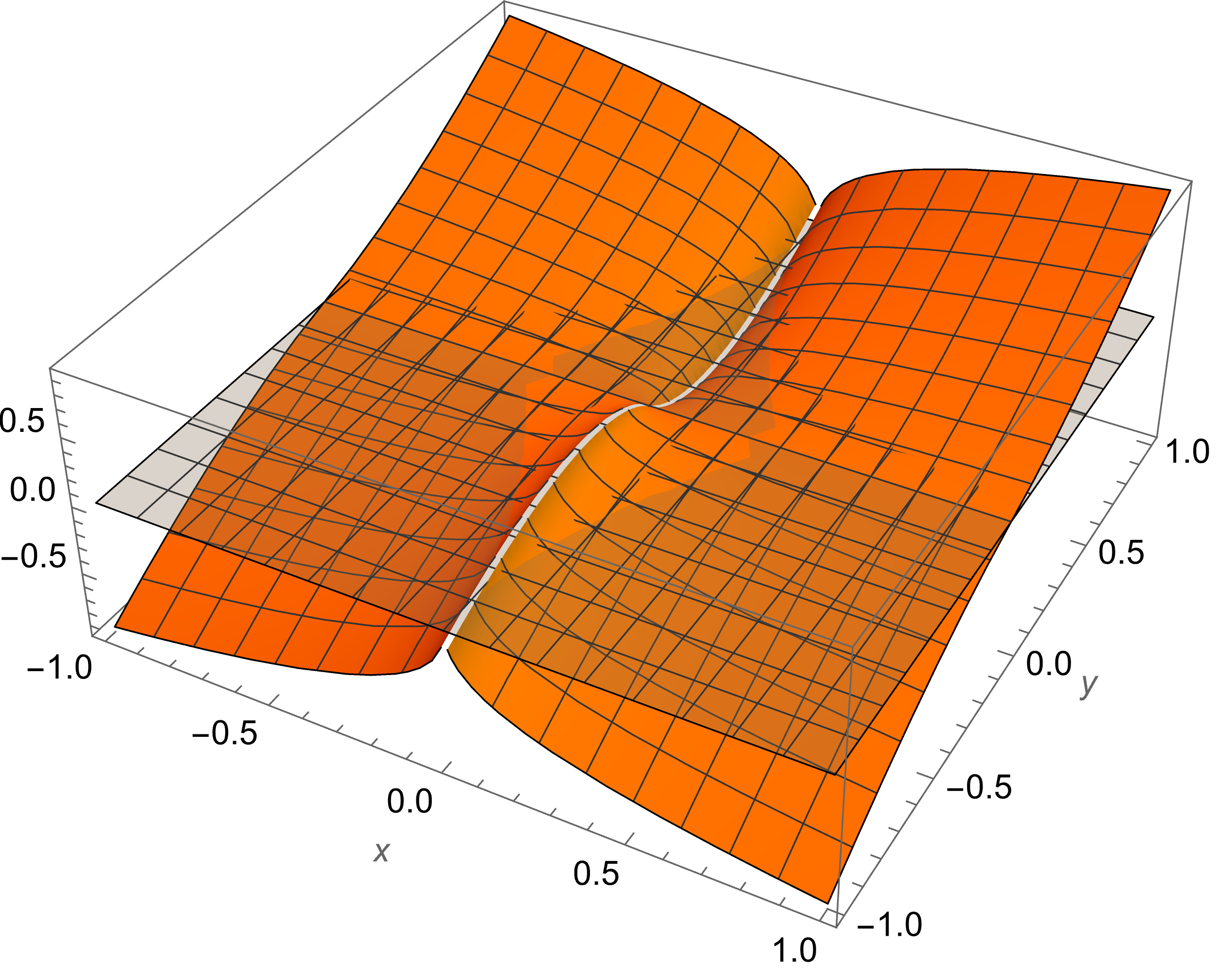}
    \caption{The function $u(x,y) = xy \log(|x|^{\frac{1}{2}}+|y|)$ (blue) compared to its $x$-derivative (orange). The $x$-derivative forms a cusp at $x=0$ which prohibits H\"{o}lder continuity beyond $\alpha = 1/2$.}
    \label{small_nu_plot}
\end{figure}
\end{remark}

\begin{remark}
The support condition in Lemma~\ref{lemma_regularity} is sufficient,
but is not necessary for every individual function. For example,
$x^2\log(x^4+y^4)$ extends as a $C^{1,\alpha}$ function for every
$\alpha<1$, although its logarithmic support is not contained in
$\operatorname{supp}(2,0)$. Indeed, its derivatives of order $k\leq2$
are bounded by $C\rho^{2-k}(1+|\log\rho|)$, where
$\rho=(x^2+y^2)^{1/2}$, and the H\"older argument above applies.
Nevertheless, dropping the support hypothesis altogether invalidates
the lemma: $u(x,y)=x^2\log(|y|^4)$ has $|n|=2$ and a smooth power
function, but $u(t,e^{-1/t^2})=-4$. It has no continuous extension
to zero. Thus some control of the logarithmic singularities by their
monomial prefactors is essential.
\end{remark}

\begin{remark}
The regularity in Lemma~\ref{lemma_regularity} cannot, in general, be improved from the
minimal coordinate regularity to the total-degree regularity $|n|-1$. Consider
\begin{equation}
        n=(1,1,1),\quad \mathfrak{q}(x,y,z)
        =
        x^{100}+y^{10000}+z^{10000},
        \qquad
        u(x,y,z)
        =
        xyz\log\mathfrak{q}(x,y,z).
\end{equation}
for $x,y,z>0$. The support condition \eqref{asslemma} is satisfied, since every variable appearing in
$\mathfrak{q}$ also appears in the monomial $xyz$. Moreover, the exponents are
large enough. However,
\begin{equation}
        \min_{j:p_j>0} n_j=1, \text{ while }
        \qquad
        |n|-1=2.
\end{equation}
Here $\nu_{\min}=100$, and the weighted index is
$$
  \Delta(n,\nu)
   =100\left(\frac1{100}+\frac1{10000}+\frac1{10000}\right)
   =\frac{51}{50}.
$$
Thus Lemma~\ref{lemma_regularity} guarantees
\begin{equation}
        u\in C^{1,\alpha},\qquad 0<\alpha<\frac1{50},
\end{equation}
whereas a statement with regularity $|n|-1$ would incorrectly predict $C^2$-regularity.\\
Indeed,
\begin{equation}
        \partial_x^2 u(x,y,z)
        =
        yz\left(
        \frac{10100x^{99}}{\mathfrak{q}(x,y,z)}
        -
        \frac{10000x^{199}}{\mathfrak{q}(x,y,z)^2}
        \right).
\end{equation}
Along the curve
\begin{equation}
        x=t,
        \qquad
        y=z=t^{1/10},
        \qquad
        t>0,
\end{equation}
we have
\begin{equation}
        \mathfrak{q}(t,t^{1/10},t^{1/10})
        =
        t^{100}+2t^{1000}
        \sim
        t^{100}.
\end{equation}
Consequently,
\begin{equation}
        \partial_x^2 u(t,t^{1/10},t^{1/10})
        \sim
        C t^{1/5}t^{-1}
        =
        C t^{-4/5},
\end{equation}
which blows up as $t\to0$. Hence $u\notin C^2$. 
This shows that, for mixed monomials, the total degree $|n|$ does not determine the
overall regularity. Both the monomial powers and the relative exponents in the power function determine the available H\"older regularity.
\end{remark}

\begin{remark}
    From a purely algebraic point of view, the subsequent analysis can be carried out equally with the basis functions
\begin{equation}\label{eta_mod_xj}
    \eta_j = \log|x_j|,\quad 1\leq j\leq d,
\end{equation}
which, in contrast to the logarithms of power functions discussed in Lemma \ref{lemma_regularity}, only guarantee $C^{r-1,1-}$-regularity of functions of the form $x^r_j\log|x_j|$, see Figure \ref{complog}.  Indeed, the basis functions $\log|x_j|$ are unbounded on a coordinate hyperplane, of real codimension one for a real coordinate and two for a complex coordinate, where $|x_j|$ is zero. Using several coordinates in a power function can improve regularity for mixed monomials by reducing its singular set. A power function supported on one coordinate gives only a constant multiple of $\log|x_j|$ and provides no such improvement.

\begin{figure}[ht]
    \centering
\includegraphics[width=0.45\linewidth]{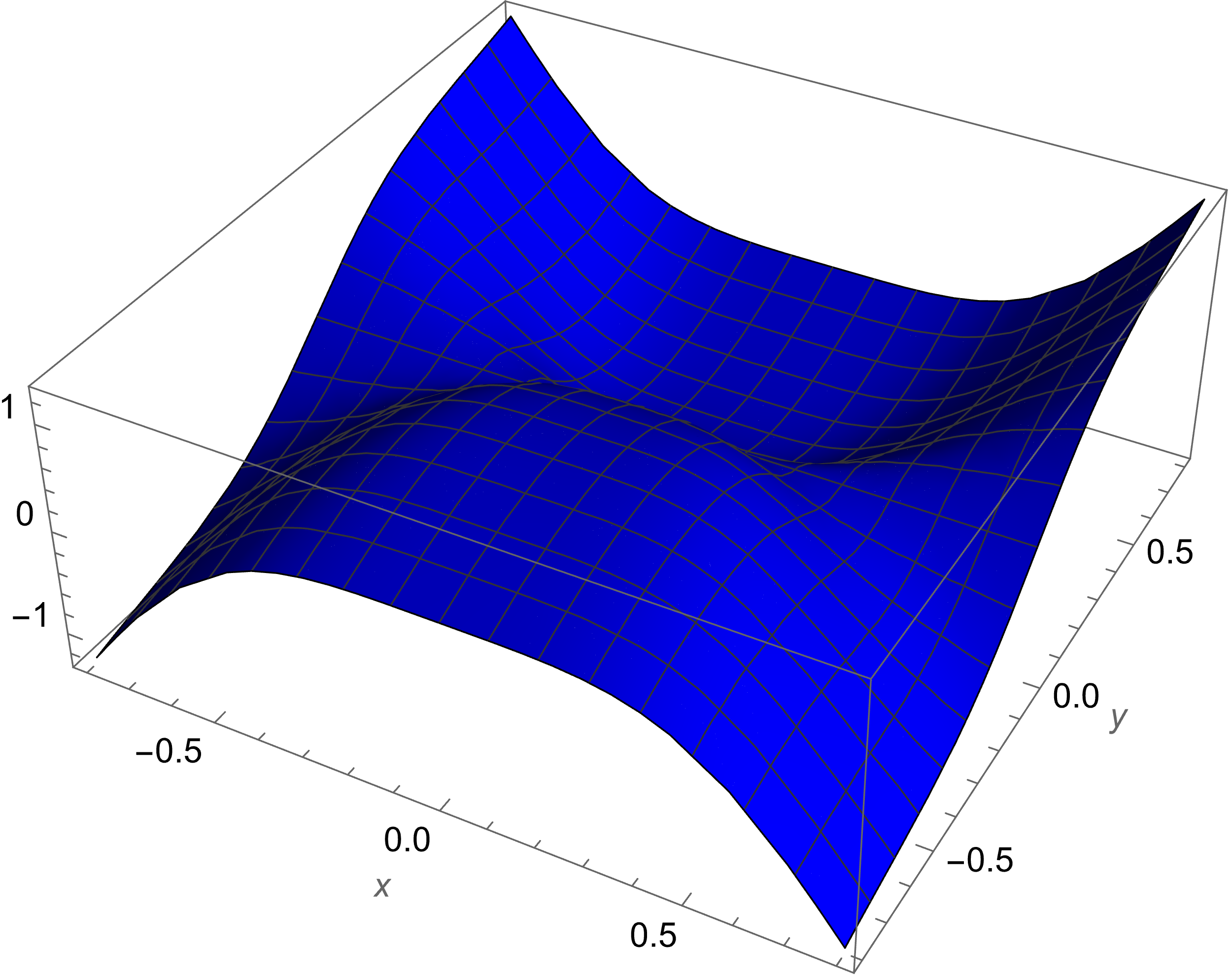}
\includegraphics[width=0.45\linewidth]{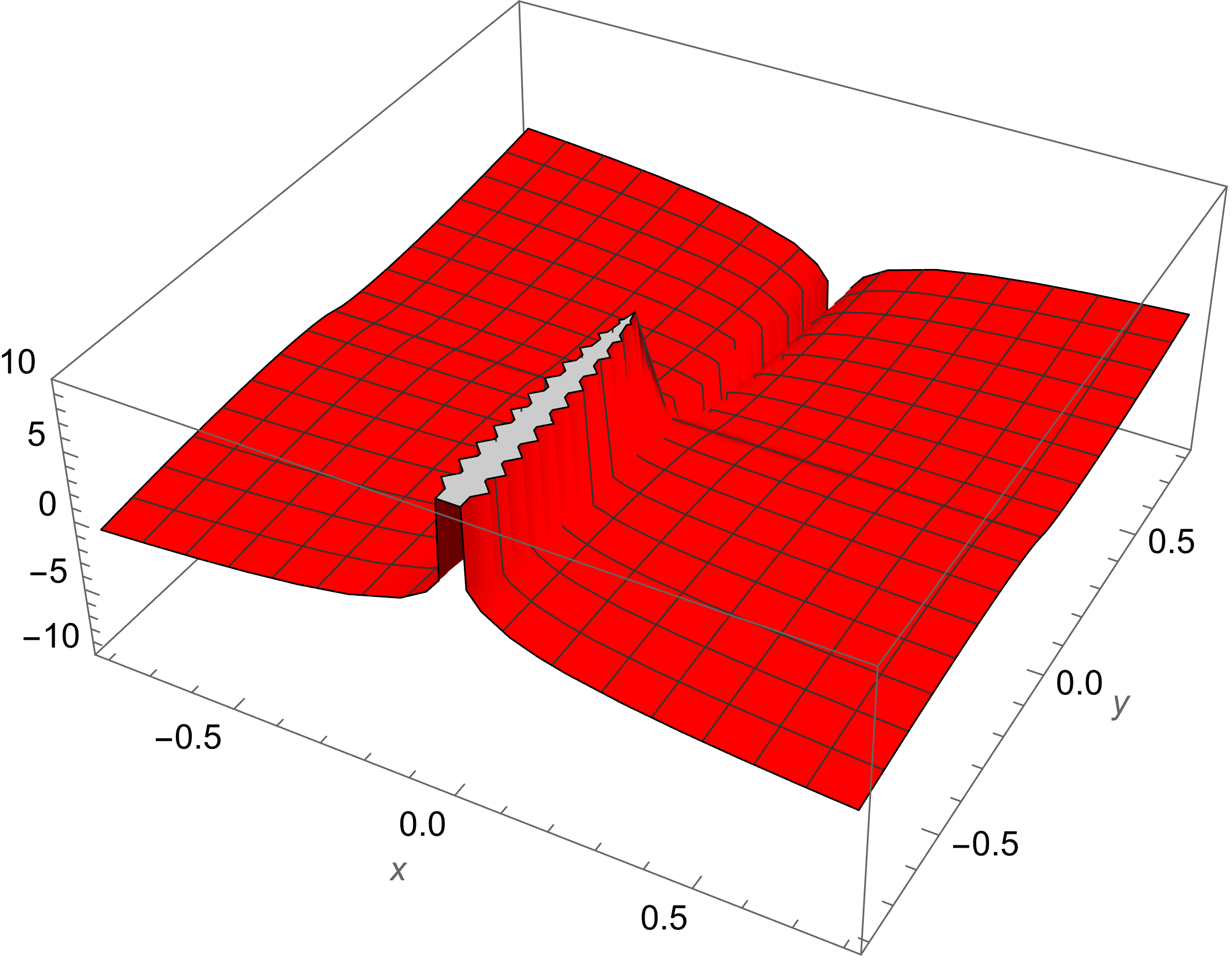}
    \caption{The $x$-derivative of the function $f(x,y) = \log(x^4+y^2)xy$ (blue) compared to the $x$-derivative of the function $g(x,y) = \log(|xy|)xy$ (red). }
    \label{complog}
\end{figure}

\end{remark}

\subsection{Approximate Solutions to the Invariance Equation}\label{approx_solutions}

We will now construct approximate solutions to the invariance equation \eqref{invariance} by means of logarithmic polynomials. Lemma \ref{shifted} will be used to guarantee that we can solve the cohomology equation order-by-order for specific logarithmic polynomials.

\begin{proposition}\label{approx}
Let $f:\mathbb C^s\to\mathbb C^s$ be analytic near zero, with
$f(0)=0$ and semisimple linear part $A=Df(0)$ satisfying
\begin{equation}\label{zeronotinspectrum}
 0\notin\sigma(A).
\end{equation}
Choose an eigenbasis $e_1,\ldots,e_s$ of $\mathbb C^s$ with
$X_1=\operatorname{span}\{e_1,\ldots,e_d\}$ and suppose
\begin{equation}\label{lambdanotone}
 |\lambda_j|<1,\qquad 1\leq j\leq d.
\end{equation}
For every $N\geq2$ there is a logarithmic polynomial
\begin{equation}\label{hN}
 h^{<N}(x)=\sum_{1\leq|n|<N}h_n(\xi)x^n,
\end{equation}
normalized by
\begin{equation}\label{pO1}
 h_{1@l}=e_l,\qquad 1\leq l\leq d,
\end{equation}
such that
\begin{equation}\label{bounddegree}
 \deg h_n\leq|n|-1,
\end{equation}
and $f\circ h^{<N}-h^{<N}\circ A$ has formal order at least $N$.
Its remainder is a formal logarithmic series
$R_N=\sum_{|n|\geq N}r_n(\xi)x^n$, not in general a finite polynomial.
Index the logarithmic variables by resonant origins $m^{(l)}$ and
choose their representatives in the form
$$
 \eta_l=\log\mathfrak{q}_l,\qquad
 \mathfrak{q}_l(x)=\sum_{j\in\operatorname{supp}(m^{(l)})}|x_j|^{\nu_j},
$$
with the common exponents \eqref{nu_form}. Every monomial
$x^n\xi^q$ occurring in $h^{<N}$ or $R_N$ satisfies
\begin{equation}\label{log_ancestor_condition}
 q_l>0\quad\Longrightarrow\quad m^{(l)}\leq n
 \quad\hbox{componentwise}.
\end{equation}
In particular, one finite collection of logarithmic variables suffices at every order.
\end{proposition}

\begin{proof}
In the chosen coordinates,
\begin{equation}\label{Ax}
 A|_{X_1}x=(\lambda_1x_1,\ldots,\lambda_dx_d).
\end{equation}
First, expand $f(z)=Az+\sum_{|a|\geq2}f_a z^a$ formally. At order $n$,
only coefficients of $h$ of lower order enter the nonlinear term:
\begin{equation}\label{expand_f}
 Q_n(\xi):=-[x^n]\sum_{|a|\geq2}f_a
       \left(\sum_{1\leq|m|<|n|}h_m(\xi)x^m\right)^a,
\end{equation}
and the coefficient extraction in \eqref{expand_f} is thus a finite sum with corresponding cohomological equation given by 
\begin{equation}\label{coholn}
 Ah_n(\xi)-\lambda_\Lambda^n h_n(\xi+v)=Q_n(\xi).
\end{equation}
Write $h_n=\sum_jh_{n,j}e_j$ and $Q_n=\sum_jQ_{n,j}e_j$ for the coordinate components, or evaluations against the dual
eigenbasis, which gives
\begin{equation}\label{eqpnj}
 \lambda_jh_{n,j}(\xi)-\lambda_\Lambda^n h_{n,j}(\xi+v)
 =Q_{n,j}(\xi).
\end{equation}
For $\lambda_\Lambda^n\ne\lambda_j$, Lemma~\ref{shifted} gives
a unique polynomial solution without increasing the logarithmic degree.
At a resonant index $n$, however, we introduce a variable $\eta_l$ with
$m^{(l)}=n$ and support $\operatorname{supp}(n)$, if it is not
already available. It has the nonzero shift $v_l=-C$.
The surjectivity in Lemma~\ref{shifted} gives a polynomial solution
of degree at most $\deg Q_{n,j}+1$, using only the old variables
occurring in $Q_{n,j}$ and this new variable. Explicitly, take
$t=\eta_l/v_l$ and the shift-invariant coordinates
$\theta_i=\eta_i-v_i t$ for those old variables. The operator
$I-T_v$ is $I-T_1$ in $t$. Apply its polynomial right inverse
coefficientwise in $\theta$, with the solution normalized to vanish
at $t=0$. The kernel consists of polynomials independent of $t$, so
this normalization fixes a linear right inverse. It raises total degree
by at most one and introduces no unused old variable. We use zero free
kernel terms for the existence construction. If other kernel terms
are retained, we require every monomial in them to satisfy
\eqref{log_ancestor_condition} and their degree to satisfy
\eqref{bounddegree}. Note that constant kernel terms are always admissible.
At a resonant index of minimal order, $Q_{n,j}$ is constant and
\begin{equation}\label{eqfirstresonance}
 h_{n,j}(\xi)-h_{n,j}(\xi+v)=\lambda_j^{-1}Q_{n,j}.
\end{equation}
Thus an affine solution $h_{n,j}=c+a\cdot\xi$ satisfies
$v\cdot a=-\lambda_j^{-1}Q_{n,j}$.\\
We verify \eqref{log_ancestor_condition} by induction. Products add
the multi-indices, so an old inequality $m^{(l)}\leq n'$ remains
valid in every product of order $n\geq n'$. Polynomial shifts and
their nonresonant inverse can only lower logarithmic powers.
At a resonant equation, the right inverse just described uses only
old variables occurring in the forcing, whose ancestors are at most
$n$, and a variable with ancestor $m^{(l)}=n$. Admissible free kernel
terms satisfy the same condition. This
preserves \eqref{log_ancestor_condition}. \\
For \eqref{bounddegree}, a nonlinear contribution of order $k$
is a product of $j\geq2$ previous terms of orders
$k_1+\cdots+k_j=k$. Its logarithmic degree is at most
$\sum_i(k_i-1)=k-j\leq k-2$. The resonant right inverse raises
the degree by at most one, while the nonresonant inverse does not raise it.
The induction starts with \eqref{pO1}. Finally,
Lemma~\ref{lemma_finite_resonances} shows that only finitely many
new logarithmic variables are introduced. For real systems choose
the bases in conjugate pairs and average paired coefficient
solutions with their real-structure conjugates at each order.
The equations, support inequalities, and degree bounds are preserved.
\end{proof}

\begin{corollary}\label{cor_approx_regularity}
For the choices made in Proposition~\ref{approx}, the approximants
$h^{<N}$ belong to $C^{r_{*},\alpha}$ throughout $0<\alpha<\alpha_{*}$
in the resonant case. If there are no resonances, these approximants
are ordinary polynomials.
\end{corollary}
\begin{proof}
Take $C$ in \eqref{nu_form} large enough that $\nu_j>r_{*}+2$.
For an active logarithmic variable in $x^n\xi^q$, condition
\eqref{log_ancestor_condition} gives, with $S_l=\operatorname{supp}(m^{(l)})$,
$$
 \left(\min_{j\in S_l}\nu_j\right)
       \sum_{j\in S_l}\frac{n_j}{\nu_j}
 \geq
 \left(\min_{j\in S_l}\nu_j\right)
       \sum_{j\in S_l}\frac{m_j^{(l)}}{\nu_j}
 =\Delta_\Lambda(m^{(l)})\geq r_{*}+\alpha_{*}.
$$
Lemma~\ref{lemma_regularity} therefore applies to every monomial.
Terms with no logarithms are polynomial. All bases and neighborhoods
can be fixed independently of $\alpha<\alpha_{*}$.
\end{proof}

\section{Proof of the Existence of Resonant Invariant Manifolds}\label{existence}

In this section, we pass from a finite logarithmic-polynomial approximation to an exact invariant manifold. We first record a coefficient-space construction: beyond the resonant orders, the linear conjugacy operator is invertible and the nonlinear correction is contractive after scaling. Section~\ref{sec_main_proof} gives a direct construction in finitely many auxiliary variables, obtaining existence and functional realization in a single contraction. \\
Traditionally, the convergence of the Taylor expansion of analytic invariant manifolds about an equilibrium is proved using the method of majorants \cite{moser1956analytic,moser2016stable}, a technique that traces back to the pioneering works of Hadamard \cite{hadamard1901} and Perron \cite{perron1929stabilitat}. A key ingredient of this approach is the simple yet crucial observation that the coefficients appearing in the Faa di Bruno formula are nonnegative \cite{comtet2012advanced,siegel2012lectures}. More recently, modern treatments reformulate the invariance equation as a fixed-point problem and establish existence results in suitably chosen function spaces \cite{figueras2017framework,hirsch1970invariant}.\\

\subsection{Invertibility of the Linear Conjugacy Operator}

For a semisimple matrix $A$ and an invariant subspace $X_1$, define $\mathcal S_A$ first on formal coefficient series, with $H(Ax)$ interpreted by the substitutions $x\mapsto A|_{X_1}x$ and $\xi\mapsto\xi+v$, as
	\begin{equation}\label{defS}
	\mathcal{S}_AH(x):=AH(x)-H(A|_{X_1}x).
	\end{equation}
The operator $\mathcal{S}_A$ is the linearization, with respect to the correction $H$, of the invariance-linearization equation around the linear dynamics and defines the vector-valued analogue of \eqref{defSav}. Its properties on the Banach space $\Gamma_N^d$ are summarized in the following lemma.

\begin{proposition}\label{Lemma_SM}
Let $A\in\mathbb C^{s\times s}$ be semisimple and invertible and let $\lambda_1,\ldots,\lambda_s$ denote the eigenvalues of $A$, counted with
multiplicity. Let
$X_1\subset\mathbb C^s$ be a $d$-dimensional $A$-invariant subspace associated with
eigenvalues $\lambda_1,\ldots,\lambda_d$, and assume
\begin{equation}
        |\lambda_i|<1,
        \qquad
        1\leq i\leq d.
\end{equation}
Set $\rho=\max_{i\leq d}|\lambda_i|<1$ and
$m=\min_{j\leq s}|\lambda_j|>0$. Suppose $N\geq2$ satisfies
\begin{equation}\label{nonres_tail}
 \rho^N<m.
\end{equation}
In particular, $\lambda_\Lambda^n\ne\lambda_j$ for $|n|\geq N$.
Assume that the fixed logarithmic variables satisfy
$\xi(Ax)=\xi(x)+v$, with $v\in\mathbb R^M$.
Then, for $\beta>0$ sufficiently small, the operator
\begin{equation}
        \mathcal S_A:\Gamma_N^d\to\Gamma_N^d,
        \qquad
        \mathcal S_A H(x)=AH(x)-H(Ax),
\end{equation}
is bounded and invertible. Moreover, there exists $q_*\in(0,1)$ such that
\begin{equation}\label{SA_inverse_bound}
        \|\mathcal S_A^{-1}\|_{\rm op}
        \le
        \frac{1}{1-q_*}
        \frac{1}{\min_{1\leq j\leq s}|\lambda_j|}.
\end{equation}
\end{proposition}

\begin{proof}
Choose coordinates on $X_1$ along eigenvectors of $A|_{X_1}$, and coordinates in
$\mathbb C^s$ along eigenvectors of $A$. Thus
\begin{equation}
        Ax=(\lambda_1x_1,\ldots,\lambda_dx_d),
        \qquad
        AH_n=(\lambda_1H_{n,1},\ldots,\lambda_sH_{n,s}).
\end{equation}
Let
\begin{equation}
        H(x)=\sum_{|n|\geq N}H_n(\xi)x^n,
        \qquad
        \widetilde H(x)=\sum_{|n|\geq N}\widetilde H_n(\xi)x^n.
\end{equation}
We consider the equation
\begin{equation}
        \mathcal S_AH=\widetilde H
\end{equation}
and assume again that the basis functions have been chosen such that $\xi(Ax)=\xi(x)+v$. The coefficient equation then becomes
\begin{equation}\label{coeff_SA}
        \lambda_j H_{n,j}(\xi)
        -
        \lambda_\Lambda^n H_{n,j}(\xi+v)
        =
        \widetilde H_{n,j}(\xi),
        \qquad
        |n|\geq N,
        \qquad
        1\leq j\leq s,
\end{equation}
or, equivalently,
\begin{equation}\label{coeff_shift}
        H_{n,j}(\xi)
        -
        a_{n,j}H_{n,j}(\xi+v)
        =
        \lambda_j^{-1}\widetilde H_{n,j}(\xi),
        \qquad
        a_{n,j}=\frac{\lambda_\Lambda^n}{\lambda_j}.
\end{equation}
Keep $N$ fixed as in \eqref{nonres_tail} and choose $\beta>0$ so small that $(\rho e^{\beta|v|})^N<m$ and $\rho e^{\beta|v|}<1$, which is possible by the strict gap assumption. Since
\begin{equation}
        \max_{1\leq i\leq d}|\lambda_i|<1,
\end{equation}
this choice gives
\begin{equation}\label{qstar_condition}
        q_*
        :=
        \sup_{\substack{|n|\geq N\\1\leq j\leq s}}
        \left|
        \frac{\lambda_\Lambda^n}{\lambda_j}
        \right|
        e^{\beta |v||n|}
        <1.
\end{equation}
Indeed,
\begin{equation}
        \left|
        \frac{\lambda_\Lambda^n}{\lambda_j}
        \right|
        e^{\beta |v||n|}
        \le
        \frac{
        \left(\max_{1\leq i\leq d}|\lambda_i|\right)^{|n|}
        e^{\beta |v||n|}
        }{
        \min_{1\leq j\leq s}|\lambda_j|
        },
\end{equation}
and the right-hand side tends to zero as $|n|\to\infty$ if $\beta>0$ is chosen so small that
\begin{equation}
        e^{\beta |v|}\max_{1\leq i\leq d}|\lambda_i|<1.
\end{equation}
By \eqref{nonres_tail}, $a_{n,j}\ne1$ for all $|n|\geq N$ and $1\leq j\leq s$.
The Neumann series $\lambda_j^{-1}\sum_{k\geq0}a_{n,j}^kT_{kv}$ converges in the weighted holomorphic coefficient space, solves \eqref{coeff_shift} and agrees with the polynomial inverse on polynomial data. It gives
\begin{equation}\label{coeff_inverse_bound}
        \|H_{n,j}\|_{\beta |n|,\sigma}
        \le
        \frac{1}{1-|a_{n,j}|e^{\beta |v||n|}}
        \frac{1}{|\lambda_j|}
        \|\widetilde H_{n,j}\|_{\beta |n|,\sigma}.
\end{equation}
To estimate the vector coefficient norm, keep the components together
in the Neumann series:
\begin{equation}\label{vector_coeff_neumann}
 H_n(\xi)=\sum_{k\geq0}(\lambda_\Lambda^n)^k
                  A^{-k-1}\widetilde H_n(\xi+kv).
\end{equation}
In the chosen ambient eigenbasis,
$\|A^{-k-1}\|_{\ell^1\to\ell^1}=m^{-k-1}$. The shift estimate
and \eqref{qstar_condition} therefore give
\begin{equation}\label{vector_coeff_inverse_bound}
\begin{split}
 \|H_n\|_{\beta|n|,\sigma}
 &\leq \frac1m\sum_{k\geq0}
   \left(\frac{|\lambda_\Lambda^n|e^{\beta|v||n|}}m\right)^k
   \|\widetilde H_n\|_{\beta|n|,\sigma}\\
 &\leq\frac1{m(1-q_*)}\|\widetilde H_n\|_{\beta|n|,\sigma}.
\end{split}
\end{equation}
Summing these vector estimates over $|n|\geq N$ yields
\begin{equation}
 \|H\|_{\Gamma_N^d,\beta,\sigma}
 \leq\frac1{m(1-q_*)}\|\widetilde H\|_{\Gamma_N^d,\beta,\sigma}.
\end{equation}
Therefore $\mathcal S_A^{-1}$ exists and satisfies \eqref{SA_inverse_bound}. Boundedness
of $\mathcal S_A$ follows directly from the definition of the norm, the finite-dimensional
boundedness of $A$, and the shift estimate for $\xi\mapsto\xi+v$. This proves the claim.
\end{proof}

\begin{remark}
    The invertibility of the operator $\mathcal{S}_A$, where $A$ is defined in \eqref{defA}, will be crucial in the reformulation of the invariance equation as a fixed-point problem in the next section. Invertibility in the completed coefficient space requires the norm gap \eqref{nonres_tail}, which in particular places every resonant index below the tail order $N$.
    For a fixed approximation and the spectral gap \eqref{nonres_tail}, the remainder is unique in the chosen coefficient ball. A matrix-norm version of Lemma~\ref{Lemma_SM} must also
control $A^{-1}$. In compatible induced norms, put
$b=\|A_1\|e^{\beta|v|}$, where $A_1=A|_{X_1}$.
The sufficient conditions $b<1$ and $\|A^{-1}\|b^N<1$ give
$$
 \|\mathcal S_A^{-1}\|
 \leq\frac{\|A^{-1}\|}{1-\|A^{-1}\|b^N}.
$$
In the chosen eigenbasis, $\|A^{-1}\|=1/m$ and
$\|A_1\|=\rho$, recovering the spectral bound. 
\end{remark}

\subsection{Reformulation of the Invariance Equation as a Fixed-Point Problem}

We will assume that $x\mapsto h(x)$ is defined on a small ball $B_{\delta,d}^\mathbb{C}$, which, after re-scaling $x\mapsto \delta x$ and setting $h^\delta(x)=h(\delta x)$, amounts to the re-scaled invariance equation:
\begin{equation}\label{coholdelta}
Ah^\delta(x)+N_f(h^\delta(x))=h^\delta(Ax).
\end{equation}
Dividing equation \eqref{coholdelta} by $\delta$ and setting $h(x)=h^\delta(x)/\delta$ by abuse of notation, the invariance equation becomes
\begin{equation}\label{invariance_expand_f}
Ah(x)+\delta N_f^\delta (h(x))=h(Ax),
\end{equation}
where
\begin{equation}\label{N_f_expand}
N_f^\delta(x)=\sum_{|n|=2}^\infty \delta^{|n|-2} f_nx^n,
\end{equation}
is the re-scaled nonlinear part of $f$. Choosing $\delta$ small makes the nonlinear coefficient majorants small on a prescribed coefficient ball. The physical neighborhood on which the resulting parametrization is realized is chosen separately in the proof of Theorem~\ref{mainthm}.\\
Assume now that we have found an approximate solution $h^{< N}$ of order $N$, according to Proposition \ref{approx} and Lemma \ref{Lemma_SM}. In the presence of resonances, where there is no longer a unique approximate solution to the invariance equation, we use the support-preserving linear right inverses in Proposition~\ref{approx} and choose the free resonant kernel terms at order $k$ to be $\delta^{k-1}$ times fixed polynomials satisfying \eqref{log_ancestor_condition} and \eqref{bounddegree}. The choice of zero kernel terms suffices for Theorem~\ref{mainthm}. Thus $h_k^\delta=\delta^{k-1}\widehat h_k$ with $\widehat h_k$ independent of $\delta$. All subsequent uniform estimates use this normalization. We write
\begin{equation}
h(x)=h^{< N}(x)+H(x),
\end{equation}
for an unknown function $H$. Since the approximation is a logarithmic polynomial of degree $N-1$, the remainder $H$ is assumed to be a formal logarithmic series starting at order $N$.
We can split the invariance equation \eqref{invariance} as
\begin{equation}
f(h^{< N}(x)+H(x))=h^{< N}(Ax)+H(Ax),
\end{equation}
which, after expanding the map $f$ according to \eqref{invariance_expand_f},
\begin{equation}\label{invariance_rewritten}
    A h^{<N}+ AH + \delta N_f^\delta\circ(h^{< N}+H) = h^{< N}\circ A+H\circ A.
\end{equation}
Regrouping and using the definition of the operator $\mathcal{S}_A$ defined in \eqref{defS}, equation \eqref{invariance_rewritten} reads 
\begin{equation}\label{coholH}
\mathcal{S}_AH = h^{< N}\circ A-Ah^{< N} -\delta N_f^\delta\circ(h^{< N}+H). 
\end{equation}
By Lemma \ref{Lemma_SM}, we can invert the operator $\mathcal{S}_A$ in the space $\Gamma_{N}^d$ and rewrite equation \eqref{coholH} as the fixed point problem
\begin{equation}
H=\mathcal{T}_\delta(H,h^{< N}),
\end{equation}
for the functional
\begin{equation}\label{defT}
\mathcal{T}_\delta(H,h^{< N}) =\mathcal{S}_A^{-1}\Big(h^{< N}\circ A-Ah^{< N} -\delta N_f^\delta(h^{< N}+H)\Big).
\end{equation}
We split the right-hand side into a fixed residual and a nonlinear tail:
\begin{equation}
        \tau_\delta
        =
        h^{<N}\circ A-Ah^{<N}
        -
        \delta N_f^\delta(h^{<N}),
\end{equation}
and
\begin{equation}
        \mathcal N_\delta(H)
        =
        -\delta\left[
        N_f^\delta(h^{<N}+H)-N_f^\delta(h^{<N})
        \right],
\end{equation}
such that the operator \eqref{defT} reads
\begin{equation}
    \mathcal{T}_\delta(H,h^{<N}) =  \mathcal{S}_A^{-1}\left(\tau_\delta+\mathcal N_\delta(H)\right)
\end{equation}

\begin{remark}
    Since the approximate solution $h^{< N}$ is not unique in the presence of resonances, the functional \eqref{defT} will vary for different approximate solutions. For a fixed approximation, the fixed point is unique in the closed coefficient ball on which the contraction is proved. Note that this is not uniqueness among all differentiable invariant manifolds.
\end{remark}

Before we state the main result of this section about fixed-points of the functional \eqref{defT}, corresponding to invariant manifolds conjugated to the linear part, we collect some functional analytic properties of $\mathcal{T}_\delta$. The following lemma is the standard regularity for affine tail composition, see also \cite{hille1996functional} for further details.

\begin{lemma}[Composition about a finite approximant]\label{lemma_affine_tail}
Let $P=h^{<N}$ have zero constant coefficient and
$\|P\|_{\Gamma_1^d}\leq B$, and let
$$
 N_f^\delta(z)=\sum_{k\geq2}\delta^{k-2}B_k(z,\ldots,z),
$$
where $B_k$ are symmetric multilinear maps. Fix $R>0$ and assume
\begin{equation}\label{composition_radius}
 C_R:=\sum_{k\geq2}k\delta^{k-2}\|B_k\|(B+R)^{k-1}<\infty.
\end{equation}
For $H\in\Gamma_N^d$ with $\|H\|\leq R$, the series
\begin{equation}\label{N_delta}
 \mathcal N_\delta(H)
 =-\delta\bigl[N_f^\delta(P+H)-N_f^\delta(P)\bigr]
\end{equation}
is well-defined and belongs to $\Gamma_N^d$. Moreover,
\begin{equation}\label{bound_difference_N_delta}
 \|\mathcal N_\delta(H_1)-\mathcal N_\delta(H_2)\|
 \leq\delta C_R\|H_1-H_2\|\qquad (\|H_i\|\leq R).
\end{equation}
If $B$ is uniform for small $\delta$, analyticity of $f$ makes
\eqref{composition_radius} uniform for $0<\delta\leq\delta_0$
after decreasing $\delta_0>0$.
\end{lemma}
\begin{proof}
The scalar algebra and vector module estimates give absolute convergence
under \eqref{composition_radius}. Each term of the difference in
\eqref{N_delta} contains a factor $H$, hence has formal order at least
$N$ (in fact at least $N+1$ for $N\geq1$). A telescoping expansion
of the $k$-linear terms at $P+H_1$ and $P+H_2$ bounds their
difference by
$k\|B_k\|(B+R)^{k-1}\|H_1-H_2\|$ and summation then proves the estimate.
For uniformity, choose $\delta_0(B+R)$ strictly within the radius
of convergence of the multilinear majorant of $f$. Note that the sum with
$\delta_0^{k-2}$ then bounds every sum with $0<\delta\leq\delta_0$.
\end{proof}

\begin{proposition}\label{propfixedpoint}
Let the assumptions of Theorem \ref{mainthm} and Lemma \ref{Lemma_SM} with
$\beta$ sufficiently small be met. Assume further that the invariance equation has been
scaled according to \eqref{coholdelta} for $\delta>0$ sufficiently small. Then, for a family $h^{<N}=\sum_{1\leq|n|<N}\delta^{|n|-1}\widehat h_n(\xi)x^n$ with the fixed admissible resonant normalization described above and with $N$ satisfying \eqref{nonres_tail}, the functional $\mathcal T_\delta(\cdot,h^{<N})$ is a contraction from the
closed unit ball $B^1(\Gamma_N^d)$ into itself and thus has a unique fixed point.
\end{proposition}

\begin{proof}
To avoid cluttering the notation, we write
\begin{equation}
        \|\cdot\|=\|\cdot\|_{\Gamma_N^d,\beta,\sigma}.
\end{equation}
By Lemma \ref{lemma_affine_tail}, $\mathcal N_\delta(H)\in\Gamma_N^d$ for $\|H\|\leq1$, and $\tau_\delta\in\Gamma_N^d$ because $h^{<N}$ solves the scaled
invariance equation up to order $N-1$. Thus $\mathcal T_\delta$ is well-defined as a map
from the closed unit ball in $\Gamma_N^d$ into $\Gamma_N^d$.

We first prove that $\mathcal T_\delta$ is a contraction on $B^1(\Gamma_N^d)$. Let
$H,\widetilde H\in B^1(\Gamma_N^d)$. By Lemma \ref{lemma_affine_tail}, applied with
$R=1$, there exists a constant $C_1>0$ such that
\begin{equation}\label{contraction_revised}
\begin{split}
        \|\mathcal T_\delta(H,h^{<N})-\mathcal T_\delta(\widetilde H,h^{<N})\|
        &=
        \left\|
        \mathcal S_A^{-1}
        \left(
        \mathcal N_\delta(H)-\mathcal N_\delta(\widetilde H)
        \right)
        \right\| \\
        &\le
        \|\mathcal S_A^{-1}\|_{\rm op}
        \|\mathcal N_\delta(H)-\mathcal N_\delta(\widetilde H)\| \\
        &\le
        \delta
        \|\mathcal S_A^{-1}\|_{\rm op}
        C_1
        \|H-\widetilde H\|.
\end{split}
\end{equation}
Choosing $\delta>0$ sufficiently small so that
\begin{equation}
        \delta
        \|\mathcal S_A^{-1}\|_{\rm op}
        C_1
        <1,
\end{equation}
we obtain that $\mathcal T_\delta(\cdot,h^{<N})$ is a contraction on
$B^1(\Gamma_N^d)$.\\
It remains to show that $\mathcal T_\delta$ maps the closed unit ball into itself. For
$H\in B^1(\Gamma_N^d)$, we have
\begin{equation}\label{ball_revised}
        \mathcal T_\delta(H,h^{<N})
        =
        \mathcal S_A^{-1}\tau_\delta
        +
        \mathcal S_A^{-1}\mathcal N_\delta(H).
\end{equation}
The approximate solution $h^{<N}$ is constructed for the scaled
invariance equation through order $N-1$. Therefore the residual $\tau_\delta$ starts at order $N$, and hence
\begin{equation}
    \tau_\delta\in\Gamma_N^d.
\end{equation}
Moreover, by \eqref{N_f_expand}, the homogeneous term of order
$k$ in the scaled nonlinear part $\delta N_f^\delta$ carries the
factor $\delta^{k-1}$. For completeness, we make the dependence on the scaling parameter $\delta$ explicit. Writing again
\begin{equation}
N_f^\delta(z)
=
\sum_{k\geq 2} B_k^\delta(z,\ldots,z),
\qquad
B_k^\delta=\delta^{k-2}B_k,
\end{equation}
we have, for $0<\delta\leq 1$, the uniform bounds
\begin{equation}
\|B_k^\delta\|\leq \|B_k\|.
\end{equation}
The finitely many coefficients of $h^{<N}$ give a uniform bound
$\|h^{<N}\|_{\Gamma_1^d}\leq B$ for $0<\delta\leq\delta_0$.
Choose $\delta_0$ so small that the majorant in
Lemma~\ref{lemma_affine_tail} converges at $B+1$.
That lemma then gives a constant $C_1$ independent of
$0<\delta\leq\delta_0$ on the closed unit tail ball. Multiplication by the prefactor $\delta$ in \eqref{N_delta} therefore gives
\eqref{bound_difference_N_delta}. Moreover, the order-by-order construction of $h^{<N}$ shows inductively that its homogeneous coefficient of order $k$ is of
the form
\begin{equation}
h_k^\delta=\delta^{k-1}\widehat h_k,
\end{equation}
with $\widehat h_k$ independent of $\delta$. Indeed, an order-$k$
contribution arising from an $m$-linear term of $N_f$, evaluated on
homogeneous terms of orders $j_1,\ldots,j_m$ with
$j_1+\cdots+j_m=k$, carries the factor
\begin{equation}
\delta^{m-1}\prod_{i=1}^m\delta^{j_i-1}
=
\delta^{k-1},
\end{equation}
and the fixed linear right inverses and the prescribed scaling of their kernel terms preserve this factor. Since all terms of orders smaller than $N$ cancel by
construction, the residual $\tau_\delta$ therefore starts at order $N$,
and its coefficients of order $k\geq N$ carry a factor $\delta^{k-1}$. To justify summation uniformly, write the residual majorant as $\sum_{k\geq N}a_k\delta^{k-1}$, where $a_k\geq0$. It converges for some $\delta_0>0$ because it is the composition of an analytic map with a finite polynomial in the graded variables. For $0<\delta\leq\delta_0$, it is bounded by $\delta^{N-1}\sum_{k\geq N}a_k\delta_0^{k-N}$. This yields
\begin{equation}\label{bound_residual_revised}
\|\mathcal S_A^{-1}\tau_\delta\|_{\Gamma_N^d}
\leq
C_\tau\delta^{N-1},
\end{equation}
for some constant $C_\tau>0$. For the nonlinear tail, Lemma \ref{lemma_affine_tail} gives, again with $R=1$,
\begin{equation}
        \|\mathcal N_\delta(H)-\mathcal N_\delta(0)\|
        \le
        \delta C_1\|H\|.
\end{equation}
Since $\mathcal N_\delta(0)=0$, this yields
\begin{equation}\label{tail_bound_revised}
        \|\mathcal S_A^{-1}\mathcal N_\delta(H)\|
        \le
        \delta
        \|\mathcal S_A^{-1}\|_{\rm op}
        C_1
        \|H\|.
\end{equation}
Combining \eqref{ball_revised}, \eqref{bound_residual_revised}, and
\eqref{tail_bound_revised}, we obtain
\begin{equation}
    \|\mathcal T_\delta(H,h^{<N})\|
    \leq
    C_\tau\delta^{N-1}
    +
    \delta\|\mathcal S_A^{-1}\|_{\rm op}C_1\|H\|.
\end{equation}
For $H\in B^1(\Gamma_N^d)$, this gives
\begin{equation}
        \|\mathcal T_\delta(H,h^{<N})\|
        \le
        C_{\tau}\delta^{N-1}
        +
        \delta
        \|\mathcal S_A^{-1}\|_{\rm op}
        C_1 .
\end{equation}
By decreasing $\delta>0$ further, if necessary, we can ensure that
\begin{equation}
        C_{\tau}\delta^{N-1}
        +
        \delta
        \|\mathcal S_A^{-1}\|_{\rm op}
        C_1
        \leq 1.
\end{equation}
Thus $\mathcal T_\delta(\cdot,h^{<N})$ maps the closed unit ball
$B^1(\Gamma_N^d)$ into itself.\\
The Banach fixed point theorem now gives a unique fixed point
$H\in B^1(\Gamma_N^d)$. Consequently,
\begin{equation}
        h=h^{<N}+H
\end{equation}
solves the scaled invariance equation. This proves the claim.
\end{proof}

\begin{remark}
Taking $d=s$ in Theorem \ref{mainthm} guarantees the existence of a local change of coordinates around zero that conjugates the full map $f$ to its linear part in the presence of resonances, provided the specified regularity and spectral properties are satisfied. We further note that, thanks to the scaling in $\delta$ in the fixed-point argument, it is enough that $N\geq 2$, but for the resolution of all the resonances, i.e., the invertibility of $\mathcal{S}_A$, we have to go to higher orders in general. Of course, the order-by-order solution of the invariance equation allows us to go to arbitrarily large $N$. 
\end{remark}

\subsection{Proof of the Main Theorem and the Stable Linearization Corollary}\label{sec_main_proof}

The following proof obtains existence and regularity together by a single contraction in an analytic algebra of finitely many complete logarithmic monomials.  \\

\begin{proof}[Proof of Theorem~\ref{mainthm}]
Suppose first that $I_{\rm res}(\Lambda)\ne\varnothing$.
The calculation in Lemma~\ref{lemma_regularity}, with \eqref{nu_form}, gives
$\Delta_\Lambda(n)>r(n)\geq r_{*}$ for every resonant index.
By Lemma~\ref{lemma_finite_resonances}, there are only finitely many
such indices, so $\alpha_{*}>0$. Choose $N$ satisfying \eqref{nonres_tail} and fix the logarithmic
bases as in Proposition~\ref{approx}, with $\nu_j>r_{*}+2$ and
zero free resonant kernel terms. For sufficiently small $\beta$
and $\delta$, Propositions~\ref{approx}, \ref{Lemma_SM}, and
\ref{propfixedpoint} give an approximation $h^{<N}$ and a fixed
point $H\in\Gamma_N^d$ such that $h^{<N}+H$ solves the scaled
invariance equation formally. It remains to justify substitution
of the logarithmic variables.\\
In the Picard iteration starting at $H=0$, the coefficient of
degree $k$ depends only on lower degrees, so each coefficient
stabilizes after finitely many iterations. The coefficients are
therefore polynomial and satisfy \eqref{bounddegree}, as in
Proposition~\ref{approx}. Moreover, every occurring monomial is a
product of monomials of $h^{<N}$ with their logarithmic powers
possibly lowered. Indeed, nonlinear composition multiplies such
monomials, while the nonresonant polynomial inverse in
Lemma~\ref{shifted} only lowers logarithmic powers.
Only finitely many possible factors are involved, all of degree
less than $N$.\\
Write $H_n(\xi)=\sum_q H_{n,q}\xi^q$. Cauchy's estimate on a
fixed polydisc inside $\mathfrak T_\sigma$, together with
$\deg H_n\leq|n|-1$, gives a constant $C_0\geq1$ such that
$$
 \sum_q|H_{n,q}|
 \leq C_0^{|n|}\|H_n\|_{|n|\beta,\sigma}.
$$
The constant is independent of $n$: the exponential weight in
\eqref{defbetanorm}, the Cauchy coefficient bounds, and the number
of logarithmic monomials all grow at most exponentially in $|n|$.
By Lemma~\ref{lemma_regularity} and the argument of
Corollary~\ref{cor_approx_regularity}, each of the finitely many
factors above belongs to every stated $C^{r_{*},\alpha}$ class
and vanishes at zero. Choose one closed ball $B$ on which each
factor $x^m\xi(x)^p$ has absolute value at most
$(2C_0)^{-|m|}$. This choice is independent of $\alpha$.
A monomial of degree $|n|$ is a product of at most $|n|$ such
factors. In the product and H\"older estimates, at most $r_{*}+1$
factors are differentiated or differenced. Since each has degree
less than $N$, this loses only a fixed number of powers of
$2C_0$. Thus every occurring monomial satisfies
$$
 \|x^n\xi(x)^q\|_{C^{r_{*},\alpha}(B)}
 \leq C_\alpha(1+|n|)^{r_{*}+1}(2C_0)^{-|n|}.
$$
Consequently,
\begin{equation}\label{analytic_lift_realization}
\begin{split}
 \sum_{|n|\geq N}
 \|H_n(\xi(x))x^n\|_{C^{r_{*},\alpha}(B)}
 &\leq C_\alpha\sum_{|n|\geq N}
 (1+|n|)^{r_{*}+1}2^{-|n|}
 \|H_n\|_{|n|\beta,\sigma}\\
 &\leq C'_\alpha\|H\|_{\Gamma_N^d,\beta,\sigma}.
\end{split}
\end{equation}
This proves absolute convergence in every claimed H\"older class
on the same ball. The estimate applies to all series generated
in the iteration, so substitution is continuous there and the
formal invariance equation becomes an identity of functions.
The same argument works with $r_{*}$ replaced by any $k\geq1$
for which the finite factors are $C^{k,\alpha}$.\\
Grouping by degree gives logarithmic-polynomial approximations
to arbitrary order, satisfying the invariance equation through
their retained orders. The support condition also gives
$|x^m\xi(x)^p|=O(|x|^{|m|-\eta})$ for each finite factor and
every $\eta>0$. For $0<\theta<1$, a smaller ball makes these
factors bounded by $(2C_0)^{-|m|}|x|^{(1-\theta)|m|}$.
Multiplication and the coefficient estimate above then bound the
tail of degree at least $K$ by $O(|x|^{(1-\theta)K})$.
Taking $\theta K<\zeta$ gives $O(|x|^{K-\zeta})$ for every
$\zeta>0$.\\
Every monomial of $H$ contains at least two of the finite factors,
since its degree is at least $N$. Absolute $C^1$ convergence
therefore gives $H(0)=DH(0)=0$. Undoing the scaling, set
$h(x)=\delta(h^{<N}+H)(x/\delta)$. Then $h(0)=0$,
$Dh(0)=\iota_{X_1}$, and $f\circ h=h\circ A|_{X_1}$.
The construction preserves the real structure as in
Proposition~\ref{approx}. If there are no resonances, all
coefficient polynomials are constant in $\xi$, and convergence
in $\Gamma_1^d$ directly gives an analytic parametrization.
Finally, the inverse-function theorem applied to the projection
of $h$ onto $X_1$ makes $h$ a local embedding. Restricting to a
small ball invariant under $A|_{X_1}$ gives forward invariance
of its image.
\end{proof}

\begin{proof}[Proof of Corollary~\ref{cor_Hartman}]
Take $X_1=\mathbb C^s$ in Theorem~\ref{mainthm}.
Then $Dh(0)=I$, so the inverse-function theorem gives a local
inverse of the same $C^{r_{*},\alpha}$ regularity, and
$$
 h^{-1}\circ f\circ h=A.
$$
Since $r_{*}\geq1$, this is a $C^{1,\alpha}$ conjugacy for every
$0<\alpha<\alpha_{*}$. The same map and a common smaller
neighborhood work for all these exponents. In the absence of
resonances, $h$ and its local inverse are analytic.
\end{proof}

\begin{proof}[Proof of Corollary~\ref{cor_weighted_regularity}]
For each resonant origin $m^{(l)}$ in Proposition~\ref{approx},
choose a nonempty $S_l\subseteq\operatorname{supp}(m^{(l)})$
maximizing $\Delta_\Lambda(m^{(l)};S_l)$, and use
$$
 \mathfrak q_l(x)=\sum_{j\in S_l}|x_j|^{\nu_j},\qquad
 \eta_l=\log\mathfrak q_l,
$$
with the exponents \eqref{nu_form}, respecting the real structure. Indeed, if the
largest contraction weight $-\log|\lambda_j|$ in a support is $t$,
adding all remaining coordinates of weight at most $t$ increases
the numerator in \eqref{def_support_degree} without changing
the denominator. We may therefore take
$S_l=\{j\in\operatorname{supp}(m^{(l)}):
-\log|\lambda_j|\leq t\}$ with the smallest maximizing $t$.
Conjugate coordinates have equal weights, so this rule respects
the involution.\\
The shift remains $v_l=-C$. Thus the proof of
Proposition~\ref{approx} applies without changing its
cohomological equations, logarithmic degree bound, or condition
\eqref{log_ancestor_condition}.
For every monomial $x^n\xi^q$ in the approximation and every
active logarithm, that condition gives
$$
 \left(\min_{j\in S_l}\nu_j\right)
       \sum_{j\in S_l}\frac{n_j}{\nu_j}
 \geq\Delta_\Lambda(m^{(l)};S_l)
 \geq D_\Lambda^{\rm opt}.
$$
Choose $C$ once so that $\nu_j>D_\Lambda^{\rm opt}+1$ for all $j$.
Lemma~\ref{lemma_regularity} then puts every complete monomial
used in the proof of Theorem~\ref{mainthm} in every claimed
$C^{k,\alpha}$ class, on one neighborhood. The same fixed-point
argument and estimate~\eqref{analytic_lift_realization}, with
$r_{*}$ replaced by $k$, give one parametrization with all these
regularities on a common smaller neighborhood.
Taking $S_l=\operatorname{supp}(m^{(l)})$ instead gives the
assertion with $D_\Lambda$. The inequalities
$D_\Lambda^{\rm opt}\geq D_\Lambda>1$ follow from
\eqref{def_optimized_degree}, finiteness of the resonant set, and
$\Delta_\Lambda(m)>r(m)\geq1$ from the main proof.
In full stable dimension, the inverse-function argument in the
proof of Corollary~\ref{cor_Hartman} gives the corresponding
regularity of the conjugacy and its inverse.
\end{proof}

\section{Examples and Counterexamples}\label{examples}

In this section, we present some examples and counterexamples that illustrate the assumptions and conclusions of our main theorems.

\begin{example}
This example illustrates the simplest resonant system exhibiting logarithmic
terms. Consider the planar differential equation
\begin{equation}\label{explres}
\begin{cases}
\dot x=-x,\\
\dot y=-2y+x^2,
\end{cases}
\end{equation}
whose eigenvalues satisfy a $(1:2)$ resonance relation. The system can be solved explicitly. For the initial condition
$(x(0),y(0))=(x_0,y_0)$ one obtains
\begin{equation}\label{sol1}
\begin{split}
x(t)&=x_0e^{-t},\\
y(t)&=\bigl(y_0+x_0^2t\bigr)e^{-2t}.
\end{split}
\end{equation}
For $x_0\ne0$, on the branch where $x/x_0>0$, eliminating the time variable via
\begin{equation}
t=-\log\!\left(\frac{x}{x_0}\right)
\end{equation}
gives the family of solution curves
\begin{equation}\label{soly1}
y=
\left[
y_0
-
x_0^2
\log\!\left(\frac{x}{x_0}\right)
\right]
\left(\frac{x}{x_0}\right)^2,
\end{equation}
or equivalently,
\begin{equation}
y
=
\frac{y_0}{x_0^2}x^2
-
x^2\log\!\left(\frac{x}{x_0}\right).
\end{equation}
Thus every trajectory with $x_0\neq0$ approaching the equilibrium, when represented as a graph over $x$, contains the resonant term
$x^2\log|x|$. The coefficient of the quadratic polynomial part depends on the
initial condition, whereas the logarithmic coefficient is fixed by the resonance. Each such graph has the form $y=Cx^2-x^2\log|x|$. Differentiating twice shows that its extension to the equilibrium is of class
$C^{1,\alpha}$ for every $\alpha<1$, but not $C^2$ at the
equilibrium, in complete agreement with the regularity predicted by the main theorem.
\end{example}

\begin{example}
Let us illustrate the non-uniqueness of the resonant invariant manifolds in a little more detail. Consider the following two-dimensional dynamical system:
\begin{equation}\label{systexpl1}
\begin{cases}
\dot{x}=-x+y^2,\\
\dot{y}=-3y + a x^2+ b x^3,
\end{cases}
\end{equation}
for some parameters $a,b \in\mathbb{R}$, whose linear part has a $1:3$ resonance.\\
For $b\neq 0$, the Taylor coefficient of $x^3$ in the $y$-dynamics is non-zero and we can capture any local resonant invariant manifold by a logarithmic expansion, while an ordinary Taylor expansion cannot satisfy the resonant equation at order three. For $b=0$, we will see that the logarithmic approximation reduces to an ordinary polynomial approximation. Since the eigenspace associated to the slow $x$-direction is just the $x$-axis, we can assume that the invariant manifold is a graph for $x$ small enough and we can thus expand $y$ as a logarithmic polynomial of the form
\begin{equation}\label{hexpl1}
\begin{split}
y(t)&=:h(x(t))\\
&=\sum_{n=1}^N h_n p_n(\log(x)) x^n(t)+R_{N+1}^{\log}(x(t)),
\end{split}
\end{equation}
for any $N\in \mathbb{N}$, $h_n\in \mathbb{R}$ and some real-valued polynomials $p_n$. Here $R_{N+1}^{\log}$ denotes formal logarithmic terms of $x$-degree at least $N+1$, not an ordinary big-$O$ remainder. We work on $x>0$. The graph equation used for coefficient comparison is $h^{\prime}(x)(-x+h(x)^2)=-3h(x)+ax^2+bx^3$.\\
 Comparison of powers shows that
\begin{equation}\label{yexpl}
\begin{split}
h(x)&=a x^2  +  (-b \log(x) + \sigma)x^3  + a^3 x^5\\
&\quad+ \left(\frac{1}{9} a^2 (21 \sigma+4 b)-\frac{7}{3} a^2 b \log (x)\right) x^6\\
&\quad +\frac{a}{8}  \left(b^2+4 b \sigma +4 b \log (x) (4 b \log (x)-b-8 \sigma )+16
\sigma ^2\right)x^7\\
&\quad +q_8(\log x)x^8+R_9^{\log}(x),
\end{split}
\end{equation}
where, writing $p=\sigma-b\log x$,
$$
 q_8(\log x)=\frac95a^5+\frac35p^3+\frac4{25}bp^2
             +\frac8{125}b^2p+\frac8{625}b^3.
$$
For any $\sigma\in\mathbb R$, the displayed truncation solves the
graph equation through order eight. For $b\ne0$, the order-eight term has logarithmic degree three,
so truncation at order seven does not give an $O(x^8)$ remainder. Note that the fifth-order term in the expansion \eqref{yexpl} does not contain any logarithmic terms, even though the third order term does. 
At zero, all the approximate invariant manifolds obtained above have the same tangency. If $b=0$, we see that, indeed, all logarithmic terms disappear. For $a=0$, the displayed terms of orders five through seven vanish, but the order-eight term generally remains. For example, $a=b=0$ gives $h(x)=\sigma x^3+\frac35\sigma^3x^8+R_9^{\log}(x)$. Figure \ref{plot1} illustrates the truncations of \eqref{yexpl} through order seven for different values of $\sigma$ as well as the polynomial approximation up to order two.
\begin{figure}
	\includegraphics[scale=1.2]{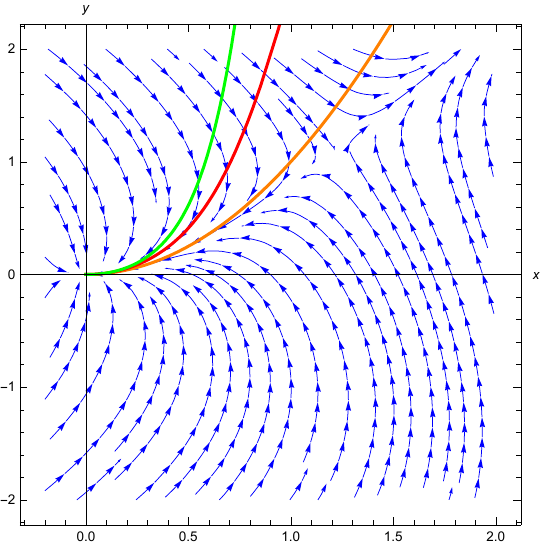}\caption{Phase portrait of system \eqref{systexpl1} for parameter values $a=b=1$ together with the truncations of \eqref{yexpl} through order seven for $\sigma=0$ (red) and for $\sigma=1$ (green). The orange line shows the polynomial approximation $h(x)=x^2$ $(N=2)$.}
	\label{plot1}
\end{figure}
\end{example}

\begin{example}
    In the next triangular flow, the H\"older threshold is sharp for invariant graphs and for all local flow conjugacies that are $C^{1,\alpha}$ diffeomorphisms, apart from the undecided endpoint. Consider the following three-dimensional differential equation,
    \begin{equation}\label{system_example}
        \begin{cases}
            \dot{{u}} & = - {u},\\
            \dot{{v}} & = -2{v},\\
            \dot{{w}} & = -3{w} + {u}{v},
         \end{cases}
    \end{equation}
with explicit solution
\begin{equation}\label{solution}
\begin{split}
    {u}(t) & = {u}_0 e^{-t},\\
    {v}(t) & = {v}_0 e^{-2t},\\
    {w}(t) & = e^{-3t}(t {u}_0 {v}_0 + {w}_0). 
\end{split}
\end{equation}
For ${u}_0{v}_0\ne0$, either of the first two identities in \eqref{solution} recovers the time: 
\begin{equation}
    t = -\log\left(\frac{{u}}{{u}_0}\right) = -\frac{1}{2}\log\left(\frac{{v}}{{v}_0}\right),
\end{equation}
which implies the representation of time as one of the basis functions 
\begin{equation}\label{rep_time}
    -\frac{1}{2}\log(|{u}|^2+|{v}|)
    =t-\frac{1}{2}\log(|{u}_0|^2+|{v}_0|).
\end{equation}
The linearity of the first two component equations leads to the invariant graph family\begin{equation}
\label{equationz}
{w} = {u}{v} \Big(A -\frac{1}{2} \log(|{u}|^2+|{v}|) \Big),
\end{equation}
for any $A\in \mathbb{R}$. Further solutions are obtained by adding $B{u}^3$ with arbitrary $B\in\mathbb R$. Figure \ref{plotz} shows the function \eqref{equationz} along with its derivatives. \\
Similarly, the flow map \eqref{solution} can be written as 
\begin{equation}
 F^t({u}_0,{v}_0,{w}_0)
 =\bigl(h\circ\diag(e^{-t},e^{-2t},e^{-3t})\circ h^{-1}\bigr)
        ({u}_0,{v}_0,{w}_0),
\end{equation}
for the conjugacy map
\begin{equation}\label{h}
    h({u},{v},{w}) = \left(\begin{array}{c}
        {u}\\ 
        {v}\\
        -\frac{1}{2}\log(|{u}^2|+|{v}|){u}{v} + {w}
    \end{array}\right).
\end{equation}
Although \eqref{h} solves the conjugacy equation, the powers in
$|{u}|^2+|{v}|$ do not meet \eqref{ass_nu_large} with $r=1$.
Direct differentiation gives the range $0<\alpha<1/2$ in
$C^{1,\alpha}$, rather than all $\alpha<1$. A representative
to which Lemma~\ref{lemma_regularity} applies is
$$
 T_4({u},{v})=-\frac18\log\bigl(|{u}|^8+|{v}|^4\bigr),\qquad
 h_4({u},{v},{w})=({u},{v},{w}+{u}{v}T_4({u},{v})).
$$
As $T_4(e^{-t}{u},e^{-2t}{v})=T_4({u},{v})+t$, this is again a conjugacy of the flow with its linear part. The associated resonant index is $n=(1,1)$, with $\nu=(8,4)$, $r_{*}=1$, and $\alpha_{*}=1/2$, so Proposition~\ref{prop_flow_counterpart} guarantees $C^{1,\alpha}$ for every $0<\alpha<1/2$. This is a sharp upper threshold for this construction: for $P_t=(t,t^2)$ and $Q_t=(t,0)$,
$$
 |P_t-Q_t|=t^2,\quad
 \bigl|\partial_{v}({u}{v}T_4)(P_t)-\partial_{v}({u}{v}T_4)(Q_t)\bigr|
 =\frac{\log2+2}{8}\,t.
$$
Therefore $h_4\notin C^{1,\alpha}$ for any $\alpha>1/2$. The endpoint $\alpha=1/2$ is left undecided by these estimates.\\
We now show that no invariant graph ${w}=g({u},{v})$, equivalently no
triangular conjugacy $({u},{v},{w})\mapsto({u},{v},{w}+g({u},{v}))$, can be
$C^{1,\alpha}$ for $\alpha>1/2$. The required equation is
\begin{equation}\label{triangular_cohomology}
 g(e^{-t}{u},e^{-2t}{v})=e^{-3t}\bigl(g({u},{v})+t{u}{v}\bigr).
\end{equation}
Assume that $g$ is $C^{1,\alpha}$ on a neighborhood of zero.
Differentiating in ${v}$ and subtracting the identity at ${v}=0$ gives
$$
 g_{v}(e^{-t}{u},e^{-2t}{v})-g_{v}(e^{-t}{u},0)
 =e^{-t}\bigl(g_{v}({u},{v})-g_{v}({u},0)\bigr).
$$
The H\"older bound implies
$$
 |g_{v}({u},{v})-g_{v}({u},0)|\leq C e^{(1-2\alpha)t}|{v}|^\alpha.
$$
Letting $t\to\infty$ shows that $g_{v}({u},{v})=g_{v}({u},0)$.
Hence $g({u},{v})={v} a({u})+b({u})$. The function $a$ is $C^1$:
for a fixed small ${v}_0\ne0$ it equals
$(g({u},{v}_0)-g({u},0))/{v}_0$.
The differentiated cohomological equation at ${v}=0$ now reads
$a(e^{-t}{u})=e^{-t}(a({u})+t{u})$.
For ${u}>0$ this forces $a({u})={u}(c-\log {u})$, whose derivative is
unbounded at zero, a contradiction.
The obstruction also applies to arbitrary local
$C^{1,\alpha}$ diffeomorphic conjugacies of the flow for $\alpha>1/2$.
Indeed, let $H(0)=0$ be such a local diffeomorphism from the linear flow to
$F^t$. Differentiating the conjugacy identity at zero shows that
$DH(0)$ commutes with $\diag(e^{-t},e^{-2t},e^{-3t})$.
The eigenvalues are distinct, so $DH(0)$ preserves the $uv$-plane.
Its image under $H$ is therefore a locally invariant $C^{1,\alpha}$
surface tangent to that plane. Projection onto $(u,v)$ has invertible
derivative there, and the $C^{1,\alpha}$ inverse-function theorem
expresses the surface as $w=g(u,v)$. This contradicts the graph
obstruction just proved. In the representation $g={u}{v}T$, the
time function $T$ is defined off its singular set and cannot have
a finite continuous value at zero. The product $g$, rather than
$T$, is the quantity extended to the origin.

\begin{figure}
    \centering
    \includegraphics[width=0.45\linewidth]{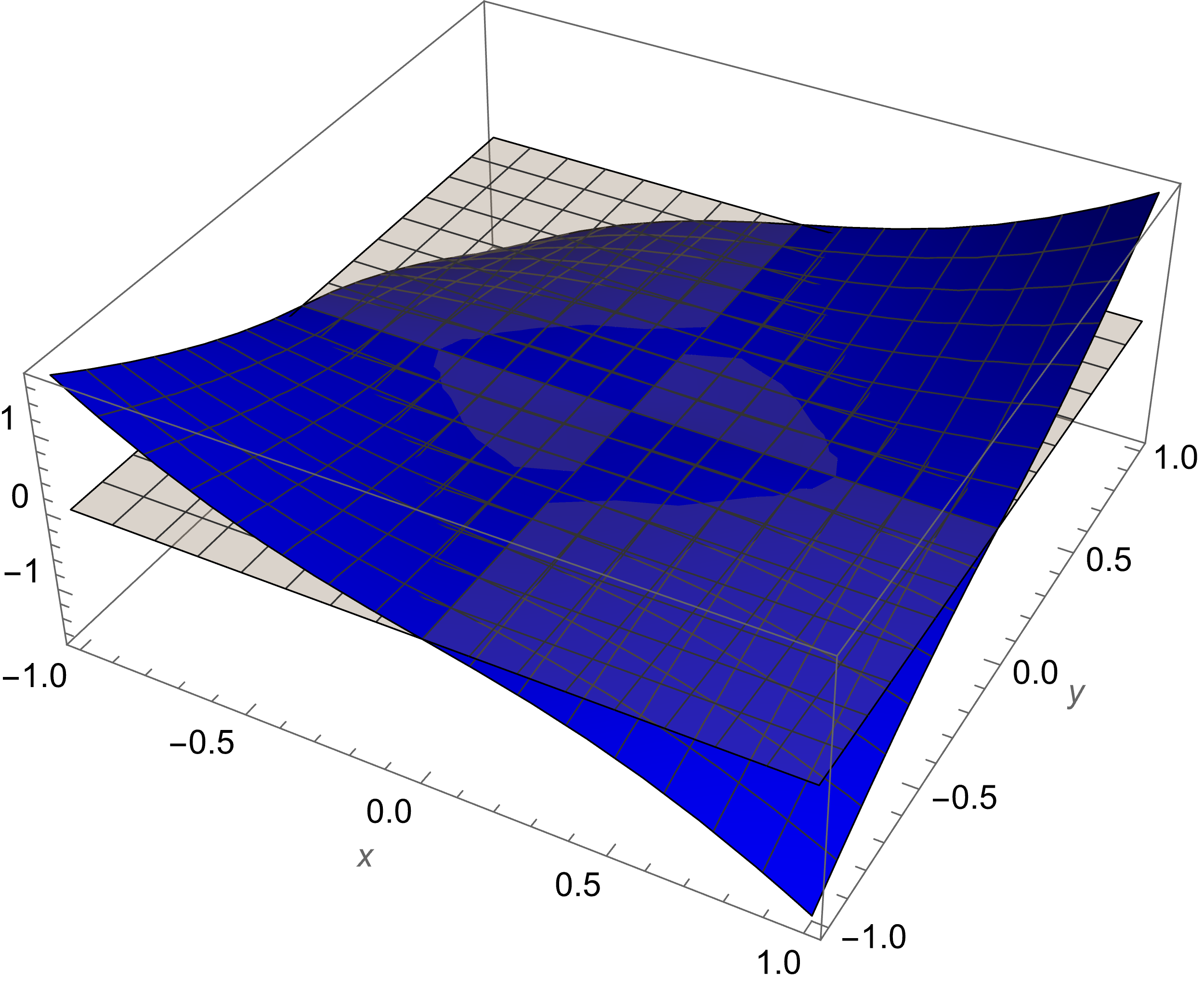}
    \includegraphics[width=0.45\linewidth]{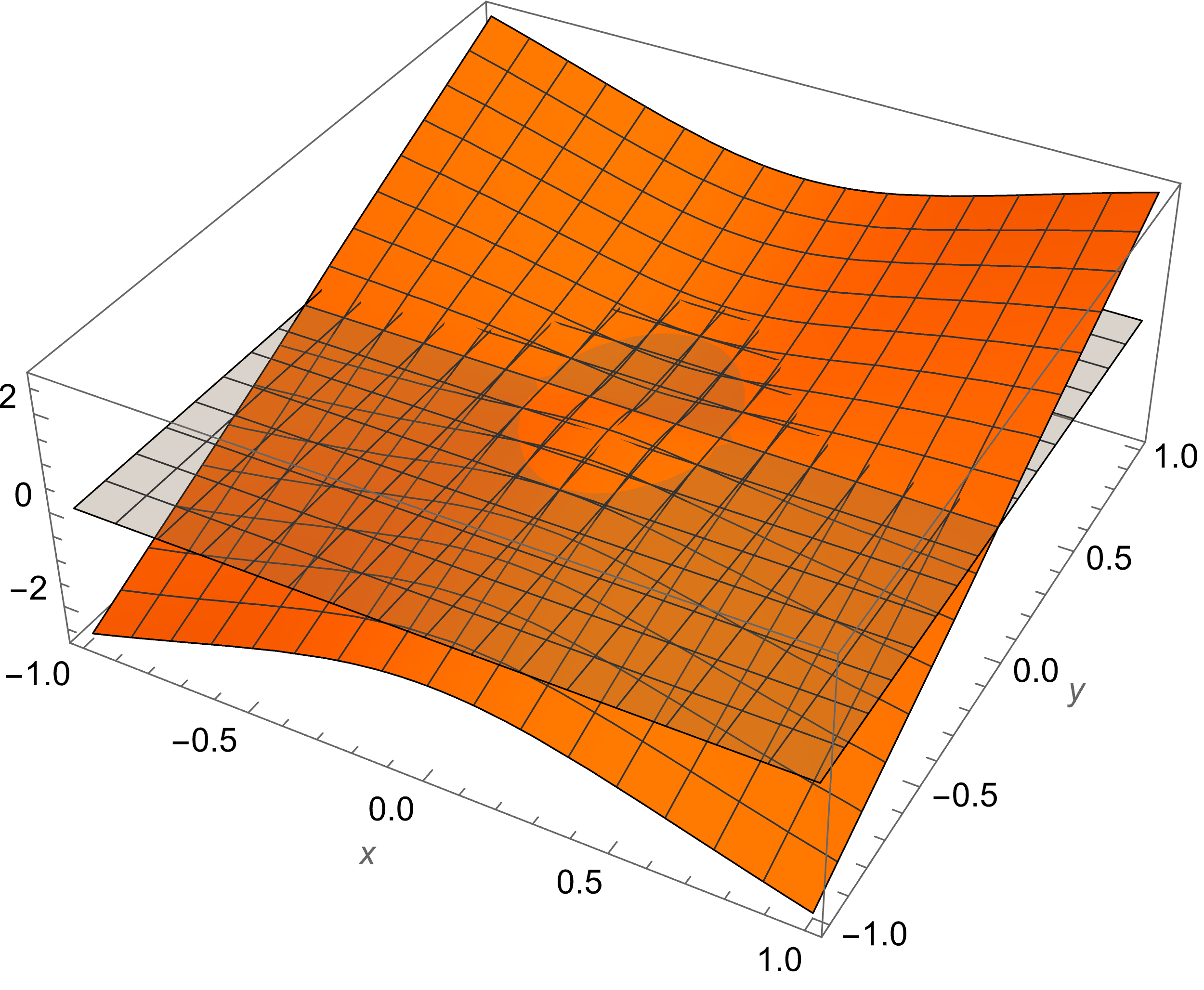}
    \includegraphics[width=0.45\linewidth]{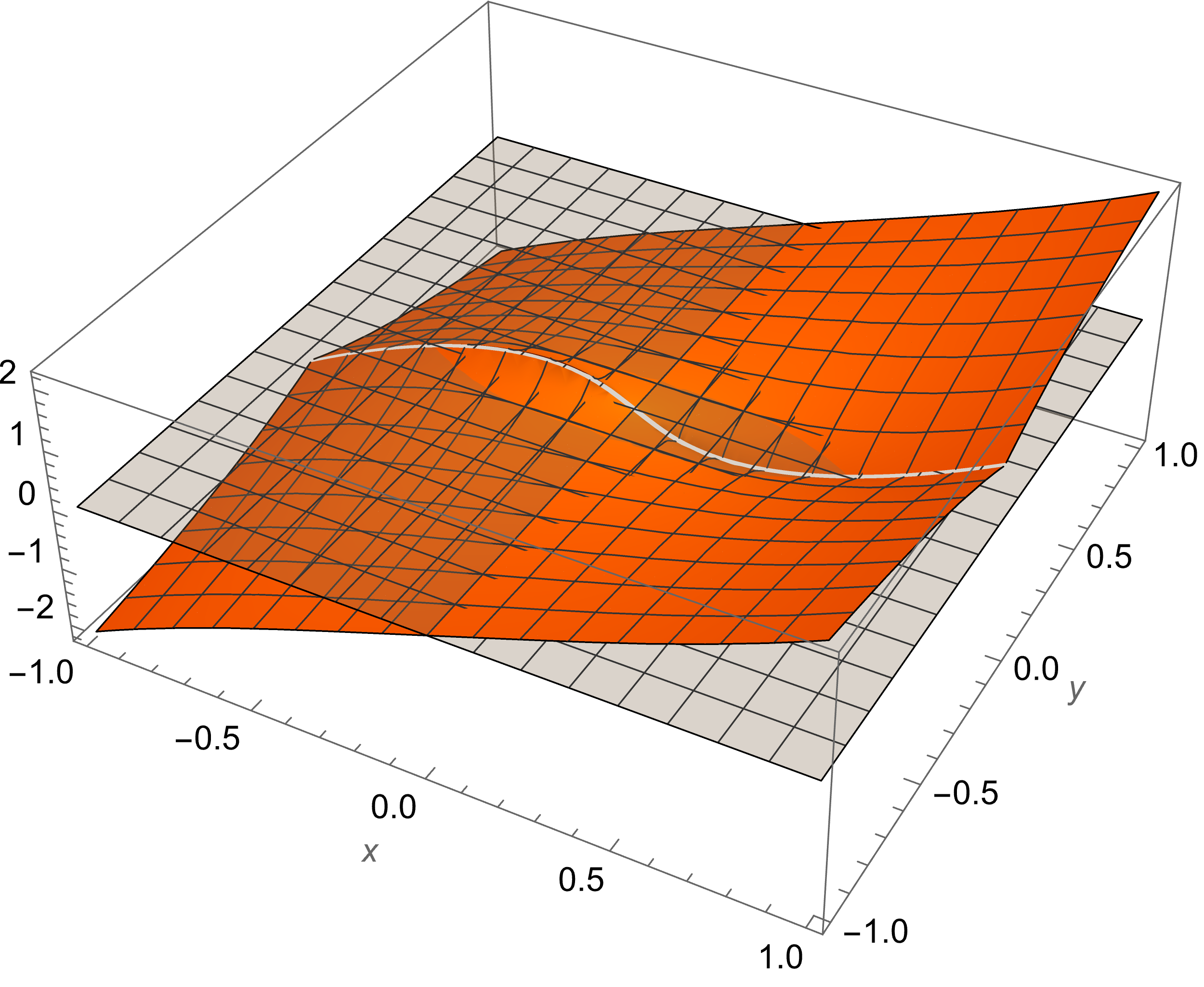}
    \includegraphics[width=0.45\linewidth]{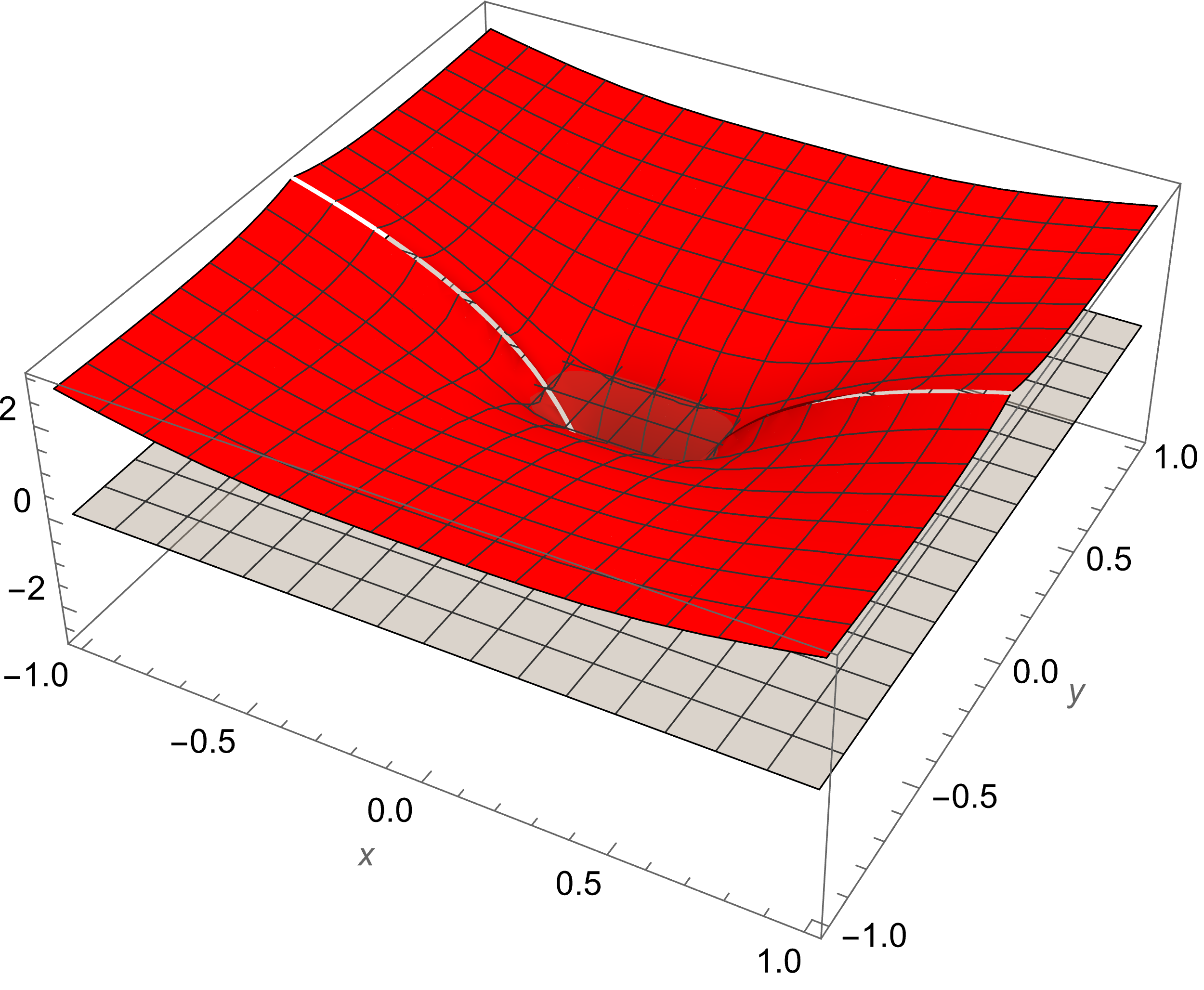}
    \caption{The invariant manifold \eqref{equationz} (blue) with its first derivatives (orange) and its ${u}{v}$-derivative (red), showing $C^{1,\alpha}$ regularity for $0<\alpha<1/2$, with failure of $C^2$ regularity}. 
    \label{plotz}
\end{figure}
\end{example}

\begin{example}[Gain from choosing the logarithmic support]
For the flow
$$
 \dot u=-u,\qquad \dot v=-100v,\qquad
 \dot w=-110w+u^{10}v,
$$
select the $uv$-plane as a reduction space. The resonant indices are
$(100,0),(110,0),(10,1)$. Thus $r_{*}=1$,
$\alpha_{*}=1/10$, and $D_\Lambda=11/10$, whereas
$D_\Lambda^{\rm opt}=10$: for $(10,1)$ the support $S=\{u\}$
gives degree $10$. The graph $g(u,v)=-u^{10}v\log|u|$, extended
by zero at $u=0$, satisfies
$$
 -u g_u-100v g_v=-110g+u^{10}v.
$$
The graph is thus  $C^{9,\alpha}$ for every $0<\alpha<1$, but not $C^{10}$
on any neighborhood of zero. This illustrates the gain in
Corollary~\ref{cor_weighted_regularity} and its flow counterpart. 
\end{example}

\begin{example}[A map with successive resonances]\label{example_successive_map}
Let $a=1/2$ and
$$
 f(x,y,z)=(ax,a^2y+x^2,a^3z+xy).
$$
Select the $x$-axis and put $\eta(x)=\log|x|/\log a$ for $x\ne0$.
Since $\eta(ax)=\eta(x)+1$, the parametrization
$$
 h(x)=\bigl(x,4x^2\eta(x),16x^3(\eta(x)^2-\eta(x))\bigr),
 \qquad h(0)=0,
$$
satisfies the invariance equation exactly: both $f(h(x))$ and
$h(ax)$ equal
$$
 \bigl(ax,x^2(\eta+1),2x^3(\eta^2+\eta)\bigr).
$$
Also $Dh(0)=(1,0,0)$ and $h\in C^{1,\alpha}$ for every
$0<\alpha<1$. Resonance at order two creates $\eta$, while the forcing
at order three creates $\eta^2$. In the formal construction with
separate variables of origins two and three, the latter term has
degree two but prefactor $x^3$. Thus successive resonances do not in general satisfy
$n\geq\sum_lq_lm^{(l)}$ componentwise.
\end{example}

\section{Conclusion and Further Perspectives}\label{conclusion}

We have proved the existence and regularity of stable invariant manifolds and linearization maps for analytic dynamical systems along strictly stable spectral subspaces of an invertible semisimple linear part, in the presence of internal and external resonances. The main conclusion is that external resonances need not be viewed only as obstructions to the invariance equation.  Rather, they indicate that the usual Taylor class is too restrictive.  Once ordinary polynomial expansions are replaced by logarithmic-polynomial expansions, the resonant cohomological equations can be solved recursively to arbitrary finite order and the resulting approximate solution can be completed by a fixed-point argument.\\

This gives a precise description of the regularity loss caused by resonance.  In the non-resonant case the normalized parametrization is analytic and is unique among analytic solutions with the prescribed tangent map. Uniqueness among less regular invariant manifolds requires an additional regularity class, as in \cite{CABRE2005444}.  In the externally resonant case one generally obtains non-unique invariant manifolds and conjugacy maps of finite differentiability, even for analytic maps. The uniform guarantee is read from the resonant indices together with the relative rates of contraction: \eqref{def_r_star}--\eqref{def_alpha_star} give $C^{r_{*},\alpha}$ for each $0<\alpha<\alpha_{*}$. The weighted-degree formulation in Corollary~\ref{cor_weighted_regularity} further gives $C^{k,\alpha}$ whenever $k+\alpha<D_\Lambda^{\rm opt}$. When the corresponding resonant forcing coefficient is nonzero, the examples show that the resulting logarithmic loss can be intrinsic, while if that coefficient vanishes, higher regularity may occur.  Logarithmic terms are therefore the natural local signature of the external resonance mechanism.\\

The proof separates the algebraic and analytic aspects of the problem.  The algebraic step consists of solving the resonant cohomology equations order by order in a logarithmic-polynomial algebra.  The analytic step uses the resulting approximate solution as the starting point for one contraction in the analytic algebra of finitely many complete logarithmic monomials. Substitution gives the exact parametrization and convergence in the stated H\"older spaces. This separation makes the method explicit and, in principle, algorithmic.  It suggests computational procedures for resonant invariant manifolds, resonant spectral submanifolds, and local conjugacy maps beyond the non-resonant regime.\\

Several prospective extensions are natural. First, one should remove some simplifying assumptions of the present work.  Although we have formulated the results in the analytic category, the resonance mechanism itself is not essentially analytic, and analogous statements should hold under sufficiently high finite smoothness.  Similarly, the assumption that the linear part is semisimple should be relaxed.  For hyperbolic matrices with Jordan blocks, the nilpotent part of the linear dynamics is expected to generate additional polynomial factors, which should interact with the logarithmic terms produced by resonances.  A complete treatment of this case would give a more general local normal form for resonant invariant manifolds.\\

Another direction is the extension from the stable setting to mixed hyperbolic spectra and to the full Hartman conjecture as well as the Van Strien conjecture \cite{vanstrien1990smooth}.  Stable and unstable spectra are formally related by inversion, but a simultaneous treatment of stable and unstable directions is more delicate \cite{hartman1960lemma,hartman1960local,newhouse2017hartman,pugh1969theorem}.  In particular, internal and external resonances may interact with both forward and backward dynamics, and some resonant terms may be absorbed into the reduced dynamics while others remain transverse obstructions.  A logarithmic-polynomial version of local conjugacy theory could provide a new approach to optimal regularity questions for Hartman--Grobman conjugacies and invariant foliations \cite{bates2000invariant,hasselblatt1994periodic,pugh1997holder,luZhangZhang2017differentiability}.\\

The construction also suggests applications in areas where resonant invariant manifolds play an organizing role.  In celestial mechanics, stable and unstable manifolds of libration-point dynamics are central in the design of low-energy transfers and in the study of resonance transitions in the restricted three-body problem \cite{koon2000heteroclinic,koon2001low,koon2011dynamical,gomez2004connecting}.  Logarithmic-polynomial parameterizations may provide more accurate local descriptions near resonant libration-point orbits and normally hyperbolic structures.  In chemical reaction dynamics, normally hyperbolic invariant manifolds and their stable and unstable manifolds organize phase-space transport and transition-state theory \cite{wiggins1994normally,komatsuzaki2001dynamical,waalkens2004direct,ezra2009microcanonical}. Since resonances affect the local normal forms near saddle-type equilibria, the present framework may help describe the asymptotic structure of resonant transition-state manifolds.\\

Another important application is nonlinear model reduction for mechanical systems.  Spectral submanifolds provide a rigorous invariant-manifold interpretation of nonlinear normal modes and an exact reduced dynamics near equilibria or periodic orbits \cite{shaw1991normal,shaw1993normal,haller2016nonlinear}.  Existing SSM methods rely heavily on polynomial parameterizations and on non-resonance or spectral-quotient conditions \cite{szalai2017nonlinear,breunung2018explicit,jain2018exact,ponsioen2020model,cenedese2022data}.  The present results suggest a resonant extension of this framework in which external resonances are handled by enlarging the parameterization class to include logarithmic terms.  This could lead to effective reduced-order models in regimes where classical polynomial SSM expansions fail or lose differentiability.\\

Finally, it would be interesting to extend the theory to functional differential equations and other infinite-dimensional systems.  Delay equations possess finite-dimensional invariant manifolds associated with spectral subspaces of the linearized semi-flow, and resonances already play a central role in their normal-form theory \cite{hale1993introduction,dieckmann1995delay,faria1995normal}.  At the same time, the Banach-space setting contains genuinely infinite-dimensional obstructions: even for contractions, $C^1$ linearization can fail without additional assumptions \cite{rodrigues2004linearization,rodrigues2005invertible,rodrigues2012known}.  A useful long-term goal is therefore to distinguish finite-dimensional resonance effects, which may be resolved by logarithmic-polynomial expansions, from obstructions caused by the spectrum and geometry of the ambient infinite-dimensional space.\\

Overall, the results show that external resonances identify the asymptotic structure that replaces analyticity.  Once this structure is incorporated into the functional setting, existence and quantitative lower regularity guarantees become accessible, opening a path toward a broader theory of resonant invariant manifolds with applications to local linearization, celestial mechanics, chemical dynamics, functional differential equations, and nonlinear model reduction.\\

\section{Acknowledgments}
The authors declare no conflict of interest. We would like to thank Tiemo Pedergnana for pointing out a mistake in a previous version of the manuscript.

\bibliographystyle{abbrv}
\bibliography{DynamicalSystems}

\appendix

\section{Continuous-time formulation and flow-map comparison}\label{sec_appendix}
Consider an analytic vector field near zero,
\begin{equation}\label{dynsyst}
 \dot x=\mathcal X(x),\qquad \mathcal X(0)=0,
\end{equation}
with
\begin{equation}\label{dyn}
 \mathcal X(x)=\mathcal Bx+\mathcal N(x),\qquad
 \mathcal B=D\mathcal X(0),\quad \mathcal N(x)=O(|x|^2).
\end{equation}
For each fixed bounded time interval its local flow $F^t$ is defined
on a sufficiently small neighborhood of zero and satisfies
\begin{equation}\label{time1}
 F^t(0)=0,\qquad DF^t(0)=e^{t\mathcal B}.
\end{equation}
Suppose that $\mathcal B$ is semisimple, with eigenvalues
$\mu_1,\ldots,\mu_s$, and $X_1$ is spanned by the first $d$
eigenvectors with $\Re\mu_j<0$ and write
$\mu_\Lambda\cdot n=\sum_{j=1}^dn_j\mu_j$. A vector-field
resonance relative to $X_1$ is
\begin{equation}\label{flowresonance}
 \mu_\Lambda\cdot n=\mu_k,
 \qquad n\in\mathbb N_0^d,\quad |n|\geq2,\quad 1\leq k\leq s.
\end{equation}
For $T>0$, the corresponding flow-map resonance is
\begin{equation}\label{mapresonance}
 e^{T\mu_\Lambda\cdot n}=e^{T\mu_k}.
\end{equation}
Note that equation \eqref{flowresonance} always implies \eqref{mapresonance}, while
additional flow-map resonances require
$T(\mu_\Lambda\cdot n-\mu_k)\in2\pi\ri\mathbb Z\setminus\{0\}$.
To ease notation, set $\lambda_j=e^{\mu_j}$, so $-\log|\lambda_j|=-\Re\mu_j>0$. Equality of real parts requires
$\sum_j n_j(-\log|\lambda_j|)=-\Re\mu_k$, which allows only finitely many
multi-indices, since $\sum_jn_j(-\log|\lambda_j|)\geq|n|\min_{j\leq d}(-\log|\lambda_j|)$.
For these finitely many pairs the exceptional positive times form
a locally finite set. For a time $T$ outside this set of exceptional times, all selected resonances then agree for the
vector field and its time-$T$ map. The ratios defining
$\Delta_\Lambda$ defined in \eqref{def_support_degree} are unchanged by replacing $-\log|\lambda_j|$ with $T(-\log|\lambda_j|)$.
Note that a partial parametrization invariant under $F^T$ need not be
invariant under all local time maps. For the logarithmic parametrizations constructed above,
the following argument also yields invariance under the full local flow.
\begin{proposition}\label{prop_flow_counterpart}
Under the preceding semisimplicity and selected stability assumptions,
the conclusions of Theorem~\ref{mainthm} and Corollary~\ref{cor_weighted_regularity} hold for a parametrization
$h$ satisfying
\begin{equation}\label{vector_invariance}
 \mathcal X(h(x))=Dh(x)\mathcal B_1x,
 \qquad \mathcal B_1=\mathcal B|_{X_1}.
\end{equation}
Here, the resonant indices are defined by \eqref{flowresonance},
$-\log|\lambda_j|=-\Re\mu_j$, and $r_{*},\alpha_{*}$ are given by
\eqref{def_r_star}--\eqref{def_alpha_star}. If there are no
resonances, $h$ is analytic and the real-structure convention is the same
as in Theorem~\ref{mainthm}. Consequently
$F^t(h(x))=h(e^{t\mathcal B_1}x)$ wherever these local expressions
are defined.
\end{proposition}

\begin{proof}
Put $a_j=-\Re\mu_j>0$ for $j\leq d$ and
$\delta_{n,k}=\mu_\Lambda\cdot n-\mu_k$. The resonance condition
for the time-$T$ map is
$$
 \prod_{j=1}^d(e^{T\mu_j})^{n_j}=e^{T\mu_k}
 \quad\Longleftrightarrow\quad
 T\delta_{n,k}\in2\pi\ri\mathbb Z.
$$
Its real part implies
$$
 |n|\min_{j\leq d}a_j
 \leq\sum_{j=1}^d n_ja_j
 =-\Re\mu_k
 \leq\max_{k\leq s}(-\Re\mu_k).
$$
Vector-field resonances ($\delta_{n,k}=0$) persist for every $T$,
and only finitely many pairs can produce additional map resonances.
Choose $T>0$ sufficiently small to exclude every nonzero multiple
of $2\pi\ri$ above, implying that the resonance pairs then coincide. Moreover,
$-\log|e^{T\mu_j}|=Ta_j$, so all the regularity indices are
unchanged.\\
Apply Theorem~\ref{mainthm} and Corollary~\ref{cor_weighted_regularity}
to $F^T$, whose derivative is the invertible semisimple matrix
$e^{T\mathcal B}$. Take the convergent formal logarithmic solution
$h$ constructed in their proofs, with full or optimized logarithmic
supports, respectively. Its logarithmic variables can be chosen with
exponents $\nu_j=C/a_j$, for sufficiently large $C>0$, so
$$
 \xi(e^{t\mathcal B_1}x)=\xi(x)-Ct\mathbf 1.
$$
If there are no resonances, $h$ is analytic and no logarithmic
variables are needed.\\
It remains to verify \eqref{vector_invariance}. Define the defect
\[
 \mathcal R(x)=\mathcal X(h(x))-Dh(x)\mathcal B_1x.
\]
The finite generators are closed under lowering logarithmic powers.
Hence differentiation along $\mathcal B_1x$ preserves their analytic
algebra, and $\mathcal R$ has a convergent logarithmic expansion on
a smaller neighborhood. Write $\widehat h(x,\xi)$ for
the formal logarithmic series before substitution, with $x$ and $\xi$
independent, and set
$$
 \mathcal D=\sum_{j=1}^d\mu_jx_j\partial_{x_j}
                -C\sum_{l=1}^M\partial_{\xi_l},\qquad
 \widehat{\mathcal R}=\mathcal X(\widehat h)-\mathcal D\widehat h.
$$
The realization of $\widehat{\mathcal R}$ is $\mathcal R$.
Put $x_t=e^{t\mathcal B_1}x$ and $\xi_t=\xi-Ct\mathbf 1$.
The time-map equation holds in this formal algebra, and all
compositions, differentiations in $t$, and integrations below are taken
coefficientwise in the spatial degree. Each coefficient is a polynomial
in $\xi$.
If $\widehat{\mathcal R}$ is not formally zero, let $q\geq2$ be its least
nonzero degree. Differentiating
$F^{T-t}(\widehat h(x_t,\xi_t))$ and integrating gives
$$
 \begin{aligned}
 0=F^T(\widehat h(x,\xi))-\widehat h(x_T,\xi_T)=\int_0^T DF^{T-t}(\widehat h(x_t,\xi_t))
            \widehat{\mathcal R}(x_t,\xi_t)\,dt.
 \end{aligned}
$$

At degree $q$, the derivative in the integrand contributes only
$e^{(T-t)\mathcal B}$. Consequently, for $|n|=q$ and each
component $k$, the coefficient polynomial $R_{n,k}$ of $\widehat{\mathcal R}$ satisfies
$$
 0=\int_0^T e^{t\delta_{n,k}}
              R_{n,k}(\xi-Ct\mathbf 1)\,dt.
$$
On polynomials of any fixed degree, this integral operator is
$\varphi_T(\delta_{n,k})I$ plus a degree-lowering, hence nilpotent,
operator, where
$$
 \varphi_T(z)=
 \begin{cases}
  (e^{Tz}-1)/z,&z\ne0,\\
  T,&z=0.
 \end{cases}
$$
Our choice of $T$ makes $\varphi_T(\delta_{n,k})\ne0$, including
at genuine vector-field resonances. The operator is therefore
invertible, forcing every $R_{n,k}$ of degree $q$ to vanish,
a contradiction. Thus $\widehat{\mathcal R}=0$ formally, and convergence
gives \eqref{vector_invariance} as an identity of functions.
All regularity, tangency, and embedding conclusions are inherited
from the map construction. Finally, uniqueness of local solutions
applied to $t\mapsto h(e^{t\mathcal B_1}x)$ gives the asserted
flow identity.
\end{proof}

\begin{remark}
In full dimension an alternative passage from a time map to the flow
uses averaging in the inverse-conjugacy convention. If
$k\circ F^T=e^{T\mathcal B}\circ k$ and $Dk(0)=I$, define
$$
 \widetilde k(x)=\frac1T\int_0^T e^{-t\mathcal B}k(F^t(x))\,dt.
$$
The integrand is $T$-periodic wherever the local identities are defined.
After restricting the neighborhood to accommodate a slightly longer
bounded time interval, changing variables gives
$\widetilde k(F^s(x))=e^{s\mathcal B}\widetilde k(x)$ for small $s$.
Also $D\widetilde k(0)=I$, so the inverse-function theorem gives a
local conjugacy of the same differentiability class. 
\end{remark}

\end{document}